\documentclass[10pt,a4paper]{amsart}
\usepackage[english]{babel}
\usepackage{amsmath,amsthm,mathrsfs,hyperref}
\usepackage{amssymb}
\usepackage{mathabx}
\usepackage{wasysym}
\usepackage{aliascnt}
\usepackage{charter}
\usepackage[utf8]{inputenc}
\usepackage[T1]{fontenc}
\usepackage{xspace}
\usepackage{comment}
\usepackage{todonotes}

\DeclareMathOperator{\dom}{dom}
\DeclareMathOperator{\range}{range}
\DeclareMathOperator{\otp}{otp}

\DeclareMathOperator{\len}{len}
\DeclareMathOperator{\crit}{crit}

\DeclareMathOperator{\rank}{rank}

\DeclareMathOperator{\cf}{cf}
\DeclareMathOperator{\stem}{stem}
\DeclareMathOperator{\Col}{Col}
\DeclareMathOperator{\E}{\mathbb{E}}

\DeclareMathOperator{\image}{\text{"}}

\newcommand{\chang}{\twoheadrightarrow}
\newcommand{\Lowenheim}{L\"{o}wenheim}
\newcommand{\Erdos}{Erd\H{o}s}

\newcommand{\ZFC}{{\rm ZFC\xspace}}
\newcommand{\GCH}{{\rm GCH\xspace}}
\newcommand{\PCF}{{\rm PCF\xspace}}
\newcommand{\guidinggeneric}{\mathcal{G}}
\newtheorem{theorem}{Theorem}
\newaliascnt{example}{theorem}

\aliascntresetthe{example}
\newaliascnt{fact}{theorem}

\aliascntresetthe{fact}
\newaliascnt{corollary}{theorem}
\newtheorem{corollary}[corollary]{Corollary}
\aliascntresetthe{corollary}
\newaliascnt{lemma}{theorem}
\newtheorem{lemma}[lemma]{Lemma}
\aliascntresetthe{lemma}

\newaliascnt{claim}{theorem}
\newtheorem{claim}[claim]{Claim}
\aliascntresetthe{claim}

\newtheorem{remark}[theorem]{Remark}
\newtheorem*{example*}{Example}

\newtheorem*{question}{Question}

\theoremstyle{definition}
\newaliascnt{definition}{theorem}
\newtheorem{definition}[definition]{Definition}
\aliascntresetthe{definition}

\begin{document}
\title{On the consistency of local and global versions of Chang's Conjecture}
\author{Monroe Eskew}
\address{Monroe Eskew \\
Kurt G\"odel Research Center, University of Vienna \\
W\"ahringer Strasse 25, 1090 Vienna \\
Austria}
\email{monroe.eskew@univie.ac.at}
\author{Yair Hayut}
\address{Yair Hayut \\
Einstein Institute of Mathematics, Hebrew University of Jerusalem \\
Jerusalem, 91904, Israel}
\email{yair.hayut@math.huji.ac.il}

\begin{abstract}
We show that for many pairs of infinite cardinals $\kappa > \mu^+ > \mu$, $(\kappa^{+}, \kappa)\chang (\mu^+, \mu)$ is consistent relative to the consistency of a supercompact cardinal. We also show that it is consistent, relative to a huge cardinal that $(\kappa^{+}, \kappa)\chang (\mu^+, \mu)$ for every successor cardinal $\kappa$ and every $\mu < \kappa$, answering a question of Foreman.
\end{abstract}
\maketitle
\section{Introduction}

The Downwards \Lowenheim-Skolem-Tarski Theorem states that every model $M$ for a language $\mathcal{L}$, $|M| = \kappa \geq \aleph_0$, and cardinal $|\mathcal{L}| + \aleph_0 \leq \mu \leq \kappa$ there is an elementary submodel $M^\prime \prec M$, of cardinality $\mu$. Informally speaking, this means that first order logic (with countable language) cannot distinguish between infinite cardinals.

Second order logic, in which we are allowed to quantify over subsets of the structure, is strong enough to distinguish between different infinite cardinals. For example, it is easy to express the statement ``There are exactly $\aleph_7$ elements in the structure'' in second order logic. By a theorem of Magidor \cite{Magidor1971}, a variant of the Downwards \Lowenheim-Skolem Theorem for full second order logic can hold only above a supercompact cardinal. In fact, there is a specific $\Pi^1_1$-formula $\Phi$ such that if $\kappa$ is a cardinal, and for every model $M$ of cardinality at least $\kappa$ that models $\Phi$ there is an elementary submodel $M^\prime$, $|M^\prime| < \kappa$, $M^\prime \models \Phi$, then there is a supercompact cardinal $\kappa_0 \leq \kappa$.

Thus, it is natural to ask how strong can be a fraction of the second order logic such that it still does not distinguish between ``most'' pairs of infinite cardinals. One candidate is first order logic enriched with the Chang's Quantifier:
\begin{definition}
Let $M$ be a model. We write \[Q x,\, \varphi(x, \vec{p})\]
if \[|\{x \in M : M \models \varphi(x, \vec{p})\}| = |M|\]
\end{definition}
We let $L(Q)$ be first order logic enriched with the quantifier $Q$. We write $M^\prime \prec_{Q} M$ if $M^\prime$ is an elementary submodel of $M$ relative to all formulas in $L(Q)$.
\begin{lemma}
The following are equivalent for infinite cardinals $\mu < \kappa$:
\begin{enumerate}
\item For every model $M$ of cardinality $\kappa^{+}$ there is an $L(Q)$-elementary submodel of cardinality $\mu^+$.
\item For every model $M$ for a language $\mathcal{L}$ that contains a predicate $A$, if $|M| = \kappa^+$ and $|A| = \kappa$, then there is $M^\prime\prec M$ with $|M^\prime| = \mu^+, |M^\prime \cap A| = \mu$.
\item For every function $f\colon (\kappa^{+})^{<\omega} \to \kappa$ there is a set $X \subseteq \kappa^{+}$, $|X| = \mu^+$, such that $|f \image X^{<\omega}| \leq \mu$.
\end{enumerate}
\end{lemma}
The second assertion is called \emph{Chang's Conjecture,} and it is denoted by:
\[(\kappa^+, \kappa) \chang (\mu^+, \mu).\]
For basic facts about Chang's Conjecture see, for example, \cite[section 7.3]{ChangKeisler1990}.

Note that if $M$ is a model of cardinality $\kappa$ of enough set theory, $M^\prime \prec_Q M$, and $|M^\prime| = \mu$, then $\kappa$ is a successor cardinal iff $\mu$ is a successor cardinal. Thus the restriction in the above lemma to successor cardinals is unavoidable.

This is not the only instance of Chang's Conjecture which provably fails. For example, assuming \GCH, for every singular cardinal $\kappa$, and every $\mu$ such that $\mu > \cf \kappa$ and $\cf \mu \neq \cf \kappa$, $(\kappa^+, \kappa) \not\chang (\mu^+, \mu)$, \cite[Lemma 1]{LevinskiMagidorShelah1990}. Without assuming $\GCH$ one can prove weaker results using Shelah's $\PCF$ mechanism: for example $(\aleph_{\omega+1},\aleph_\omega)\chang (\aleph_{n+1}, \aleph_n)$, implies that for every scale on $\aleph_\omega$ there are stationarily many bad points of cofinality $\omega_{n+1}$ (see \cite{ForemanMagidor1997}). It is provable that for every scale on $\aleph_{\omega}$ there is a club in which every ordinal of cofinality $\geq \omega_4$ is good (see, for example, \cite{AbrahamMagidor2010}), and therefore $(\aleph_{\omega+1},\aleph_\omega)\not\chang (\aleph_{n+1}, \aleph_n)$ for every $n \geq 3$. The consistency of the cases $(\aleph_{\omega+1},\aleph_\omega)\chang (\aleph_2, \aleph_1)$ and $(\aleph_{\omega+1},\aleph_\omega)\chang (\aleph_3, \aleph_2)$ is completely open.

There are many open questions about Chang's Conjecture at successors of singular cardinals. In this paper we will concentrate on the questions about Chang's Conjecture at successors of regular cardinals.
In Section \ref{sec: local cc}, we will show how to force instances of Chang's Conjecture of the form $(\kappa^{+},\kappa)\chang (\mu^+,\mu)$, where $\kappa > \mu^+$ for various values of $\kappa$ and $\mu$ from large cardinals weaker than a supercompact cardinal. This improves the known upper bounds for the consistency strength of those instances.

In Section \ref{sec: global cc} we will show that it is consistent relative to a huge cardinal that for every successor $\kappa$ and every $\mu < \kappa$, $(\kappa^+,\kappa)\chang (\mu^+,\mu)$. This answers a question of Foreman from \cite[Section 12, Question 7]{Foreman2010ideals}.

In Section \ref{sec: segment cc} we will construct a model in which for all $m < n < \omega$, $(\aleph_{n+1}, \aleph_n)\chang(\aleph_{m+1}, \aleph_m)$ and $(\aleph_{\omega+1},\aleph_{\omega})\chang (\aleph_1, \aleph_0)$.

Our notations are mostly standard. We work in \ZFC, and specify any usage of large cardinals. For basic facts about forcing see \cite{Jech2003}. For standard facts and definitions about large cardinals, see \cite{Kanamori2009}.
\section{Preliminaries and preservation lemmas}\label{sec: preliminaries}
We begin with some standard notations and definitions.
\begin{definition}[Levy collapse]
Let $\kappa < \lambda$ be cardinals. $\Col(\kappa, \lambda)$ is the set of all partial functions, $f\colon \kappa \to \lambda$, $|\dom f| < \kappa$, ordered by reverse inclusion.
\end{definition}
$\Col(\kappa,\lambda)$ adds a surejction from $\kappa$ onto $\lambda$. If $\kappa$ is regular cardinal, this forcing is $\kappa$-closed. The forcing $\Col(\kappa, {<}\lambda)$ is the product with support ${<}\kappa$ of $\Col(\kappa, \alpha),\ \alpha < \lambda$. If $\lambda$ is inaccessible then this forcing is $\lambda$-c.c.

\begin{definition}[Easton Collapse, Shioya]\cite{Shioya2011}
Let $\kappa < \lambda$ be cardinals. $\E(\kappa,\lambda)$ is the Easton-support product of $\Col(\kappa, \alpha)$, over inaccessible $\alpha \in (\kappa,\lambda)$.
\end{definition}
\begin{definition}
A partial order $\mathbb P$ has \emph{precaliber} $\kappa$ if for every $A \subseteq \mathbb P$ of size $\kappa$, there is $B \subseteq A$ of size $\kappa$ that generates a filter.
\end{definition}
\begin{lemma}
Assume that $\kappa$ is regular and $\lambda$ is Mahlo. Then $\E(\kappa,\lambda)$ has precaliber $\lambda$ and is $\kappa$-closed.
\end{lemma}

The following absorption lemmas will be important in our use of huge cardinal embeddings.

\begin{lemma}[Folklore]
\label{folk}
If $\mathbb P$ is $\kappa$-closed and forces $|\mathbb P| = \kappa$, then $\mathbb P$ is forcing-equivalent to $\Col(\kappa,|\mathbb P|)$.
\end{lemma}

The following notion and lemma are due to Laver, which we will show in somewhat more generality in Lemma \ref{easton projection}.

\begin{definition}Suppose $\mathbb P$ is a partial order and $\dot{\mathbb Q}$ is a $\mathbb P$-name for a partial order.  The termspace forcing for $\dot{\mathbb Q}$, denoted $T(\mathbb P,\dot{\mathbb Q})$, is the set of equivalence classes of $\mathbb P$-names of minimal rank for elements of $\dot{\mathbb Q}$.  Two names $\tau,\sigma$ are equivalent when $1 \Vdash \tau = \sigma$.  The ordering is $[\tau] \leq [\sigma]$ iff $1 \Vdash_{\mathbb P} \tau \leq \sigma$.
\end{definition}

\begin{lemma}If $\dot{\mathbb Q}$ is a $\mathbb P$-name for a partial order, then the identity map from $\mathbb P \times T(\mathbb P,\dot{\mathbb Q})$ to $\mathbb P * \dot{\mathbb Q}$ is a projection.
\end{lemma}

In this paper, we will use \Erdos-Rado theorem repeatedly. Since, at some points, we will need to refer to its proof we provide here a proof as well for the case that interests us.
\begin{theorem}[\Erdos-Rado]\label{thm: Erdos-Rado}\cite{ErdosRado1956}
Let $\kappa$ be a regular cardinal and let $\rho < \kappa$.
\[(2^{{<}\kappa})^+\to (\kappa + 1)^2_\rho\]
Namely, for every function $f$ from the unordered pairs of $(2^{{<}\kappa})^+$ to $\rho$ there is a set $H$ of order type $\kappa+1$ and an ordinal $\alpha$ such that for all $x, y\in H$, $f(x, y) = \alpha$.
\end{theorem}
\begin{proof}
Let $\lambda = (2^{{<}\kappa})^+$. Let $M$ be an elementary submodel of $H(\chi)$ (for large enough $\chi$), with $f\in M$, $^{{<}\kappa} M \subseteq M$, $M \cap \lambda$ is an ordinal and $|M| = 2^{{<}\kappa}$.

Let $\delta = M \cap \lambda$. By the closure assumptions on $M$, $\cf \delta \geq \kappa$.
Let us construct, by induction, an increasing sequence of ordinals $\alpha_i \in M$ such that:
\[\forall i < j < \kappa,\, f(\alpha_i, \alpha_j) = f(\alpha_i, \delta)\]

Assume $\eta < \kappa$ and we have constructed the first $\eta$ members of the sequence $\langle \alpha_i \colon i < \eta \rangle$. The element $r = \langle f(\alpha_i, \delta) \colon i < \eta\rangle\in \,^\eta \rho$ belongs to $M$, as $M$ is closed under ${<}\kappa$-sequences and thus contains $H(\kappa)$. Similarly, the function $g\colon \eta \to M \cap \lambda$, $g(i) = \alpha_i$, belongs to $M$.

\[H(\chi) \models \exists \zeta,\, \forall i < \eta,\, f(g(i), \zeta) = r(i),\, \zeta > \sup g\image \eta\]
as witnessed by $\zeta = \delta$. By elementarity, the same holds in $M$, so there is $\zeta\in M$ such that $f(\alpha_i, \zeta) = r(i) = f(\alpha_i, \delta)$ for all $i < \eta$. Take $\alpha_\eta = \zeta$.

Next, let us narrow down the sequence $\langle \alpha_i \colon i < \kappa\rangle$ to a homogeneous set. Let $\rho_i = f(\alpha_i, \delta) < \rho$. Since $\kappa$ is regular and $\rho < \kappa$, there is some $\rho_\star < \kappa$ and an unbounded set $I \subseteq \kappa$ such that for all $i\in I$, $\rho_i = \rho_\star$. Let $H = \{\alpha_i \colon i \in I\} \cup \{\delta\}$. For every $\alpha < \beta$ in $H$, $f(\alpha, \beta) = f(\alpha,\delta) = \rho_\star$, as wanted.
\end{proof}

\begin{remark}\label{remark: erdos-rado for 2k colors}
The same proof shows that whenever $\kappa$ is regular, $\lambda^{<\kappa} = \lambda$, and \[f\colon [\lambda^+]^2 \to \lambda,\] there is an increasing sequence of ordinals $\langle \alpha_i \colon i < \kappa + 1\rangle$, $\delta = \alpha_{\kappa}$ such that $f(\alpha_i, \alpha_j) = f(\alpha_i, \delta)$ for all $i < j < \kappa$. This observation will come in handy later.
\end{remark}

In this paper, we are interested in transfer properties between pairs of cardinals.  However, consideration of transfer between larger collections of cardinals will aid in the investigation of pairs.  Suppose $\langle \lambda_i \rangle_{i \in I}$, $\langle \kappa_i \rangle_{i \in I}$ are sequences of cardinals.  The notation
\[ \langle \lambda_i \rangle_{i \in I} \chang \langle \kappa_i \rangle_{i \in I} \]
signifies the assertion that for every $f : [\lambda]^{{<}\omega} \to \lambda$, where $\lambda = \sup \lambda_i$, there is $X \subseteq \lambda$ closed under $f$ such that $|X \cap \lambda_i| = \kappa_i$ for each $i \in I$.

Let us note a few easy facts about these principles:
\begin{itemize}
\item If $J \subseteq I$ and $\langle \lambda_i \rangle_{i \in I} \twoheadrightarrow \langle \kappa_i \rangle_{i \in I}$, then $\langle \lambda_i \rangle_{i \in J} \twoheadrightarrow \langle \kappa_i \rangle_{i \in J}$.
\item If $\langle \lambda_i \rangle_{i \in I} \twoheadrightarrow \langle \kappa_i \rangle_{i \in I}$ and $\langle \kappa_i \rangle_{i \in I} \twoheadrightarrow \langle \mu_i \rangle_{i \in I}$, then $\langle \lambda_i \rangle_{i \in I} \twoheadrightarrow \langle \mu_i \rangle_{i \in I}$.

\item If $\lambda_j,\kappa_j$ are the maximum elements of $\langle \lambda_i \rangle_{i \in I}$, $\langle \kappa_i \rangle_{i \in I}$ respectively, $\lambda' > \lambda_j$, and $\kappa_j > \kappa' \geq \sup_{i \not= j} \kappa_i$, then $\{ (j,\lambda') \} \cup \langle \lambda_i \rangle_{i \not= j} \twoheadrightarrow \{ (j,\kappa') \} \cup \langle \kappa_i \rangle_{i \not= j}$
\end{itemize}

During the construction of our models, we would like to use the fact that Chang's Conjecture is indestructible under a wide variety of forcing notions.

\begin{definition}
If $\lambda_1 \geq \lambda_0$ and $\kappa_1 \geq \kappa_0$ are cardinals and $\xi$ is an ordinal, let
\[ (\lambda_1,\lambda_0) \chang_\xi (\kappa_1,\kappa_0) \]
stand for the statement that for all $f : \lambda_1^{{<}\omega} \to \lambda_1$, there is $X \subseteq \lambda_1$ of size $\kappa_1$ such that $f\image (X^{{<}\omega}) \subseteq X$, $|X \cap \lambda_0| = \kappa_0$, and $\xi \subseteq X$.
\end{definition}

\begin{lemma}[Folklore]
\label{preserve small}
The statement $(\lambda_1,\lambda_0) \chang_{\kappa_0} (\kappa_1,\kappa_0)$ is preserved by $\kappa_0^+$-c.c.\ forcing.
\end{lemma}
\begin{proof}
Suppose $\mathbb P$ is a $\kappa_0^+$-c.c.\ forcing and $\dot f$ is a $\mathbb P$-name for a function from $\lambda_1^{{<}\omega}$ to $\lambda_1$.  For every $x \in \lambda_1^{{<}\omega}$, let us look at the $\mathbb{P}$-name $\dot{f}(x)$. By the chain condition of $\mathbb{P}$, the set of possible values for $\dot f(x)$ has size $\leq \kappa_0$.  If $g(\alpha,x)$ returns the $\alpha^{th}$ possible value for $\dot f(x)$, then a set closed under $g$ of the appropriate type will be closed under $\dot f$ in the extension by $\mathbb P$.
\end{proof}

\begin{lemma}
Suppose either $\lambda_0 = \kappa_0^{+\xi}$, or there is $\lambda \leq \lambda_0$ such that $\lambda_0 = \lambda^{+\xi}$ and $\lambda^{\kappa_0} \leq \lambda_0$.  If $(\lambda_1,\lambda_0) \chang_\xi (\kappa_1,\kappa_0)$, then $(\lambda_1,\lambda_0) \chang_{\kappa_0} (\kappa_1,\kappa_0)$.
\end{lemma}
\begin{proof}
See \cite[Section 2.2.1]{Foreman2009}.
\end{proof}
Under \GCH, in many cases $(\lambda_1,\lambda_0) \chang (\kappa_1,\kappa_0)$ implies $(\lambda_1,\lambda_0) \chang_{\kappa_0} (\kappa_1,\kappa_0)$ (see \cite[Proposition 19]{Foreman2009}). For example, this is true for regular cardinals $\kappa_0, \kappa_1, \lambda_0, \lambda_1$.
Foreman \cite{Foreman1983} proved the next result for the case where $\mathbb{P}$ is trivial.

\begin{lemma}\label{preserve closed}
Let $\kappa_1$ be regular. Suppose $\mathbb P$ is $\kappa_1$-c.c., $\mathbb Q$ is $\kappa_1$-closed, and $\Vdash_{\mathbb P} \dot{\mathbb R} \lhd \check{\mathbb Q}$.  If $\Vdash_{\mathbb P} (\kappa_1,\kappa_0) \chang (\mu_1,\mu_0)$, then $\Vdash_{\mathbb P * \dot{\mathbb R}} (\kappa_1,\kappa_0) \chang (\mu_1,\mu_0)$.
\end{lemma}

\begin{proof}
First, let us claim that if $\mathbb P \times \mathbb Q$ forces that  $(\kappa_1,\kappa_0) \chang (\mu_1,\mu_0)$ holds then also $\Vdash_{\mathbb P \ast \dot{\mathbb R}} (\kappa_1,\kappa_0) \chang (\mu_1,\mu_0)$.  

By Easton's lemma, $\Vdash_{\mathbb P}$ ``$\mathbb Q$ is $\kappa_1$-distributive,'' thus $\Vdash_{\mathbb P * \dot{\mathbb R}}$ ``$\mathbb Q / \mathbb R$ is $\kappa_1$-distributive.''  Let $G * \dot H$ be $\mathbb P * \dot{\mathbb R}$-generic, and let $f : \kappa_1^{<\omega} \to \kappa_1$ be in $V[G][H]$.  If $H^\prime$ is $\mathbb Q / H$-generic, then in $V[G][H][H^\prime]$, there is an $X \subseteq \kappa_1$ closed under $f$ such that $|X| = \mu_1$ and $|X \cap \kappa_0| = \mu_0$.  By the distributivity of $\mathbb Q / \mathbb R$, $X \in V[G][H]$.

Now suppose $\Vdash_{\mathbb P \times \mathbb Q}$ ``$\dot f : \kappa_1^{<\omega} \to \kappa_1$.''  Let $\langle s_\alpha : \alpha < \kappa_1 \rangle$ enumerate $\kappa_1^{<\omega}$, and let $(p,q) \in \mathbb P \times \mathbb Q$ be arbitrary.  Let $(p^0_0,q^0_0)$ decide $\dot f(s_0)$, with $(p^0_0,q^0_0) \leq (p,q)$.  If possible, choose $p^1_0 \perp p^0_0$ also below $p$, and let $q^1_0 \leq q^0_0$ be such that $(p^1_0,q^1_0)$ decides $\dot f(s_0)$.  For as long as possible, keep choosing a sequence of pairs $(p_0^\alpha,q_0^\alpha)$, such that the $p_0^\alpha$'s form an antichain below $p$ and the $q_0^\alpha$'s form a descending sequence.  If we have chosen less than $\kappa_1$ $q^0_\alpha$'s, we can choose a condition below all of them by $\kappa_1$-closure.  By the $\kappa_1$-c.c., this must terminate at some $\eta_0 < \kappa_1$, at which point $\langle p_0^\alpha : \alpha < \eta_0 \rangle$ is a maximal antichain below $p$.  Let $q_0^* \leq q_0^\alpha$ for all $\alpha < \eta_0$.  Next, do the same with respect to $\dot f(s_1)$, but starting with $q^0_1 \leq q^*_0$.  Continuing in this way, we get maximal antichains $\langle p^\beta_\alpha : \beta < \eta_\beta \rangle$ in $\mathbb P \restriction p$ for each $\alpha < \kappa_1$, and a descending sequence $\langle q^*_\alpha : \alpha < \kappa_1 \rangle$, with the property that whenever $\alpha < \kappa_1$, $\beta < \eta_\alpha$ and $\alpha \leq \gamma < \kappa_1$, $(p^\beta_\alpha,q^*_\gamma)$ decides $\dot f(s_\alpha)$.

Let $G \subseteq \mathbb P$ be generic over $V$ with $p \in G$.  For each $\alpha < \kappa_1$, there is a unique $\beta < \eta_\alpha$ such that $p^\beta_\alpha =_{\mathrm{def}} p^*_\alpha \in G$.  In $V[G]$, define a function $f^\prime : \kappa_1^{<\omega} \to \kappa_1$ by $f^\prime(s_\alpha) = \beta$ iff $(p^*_\alpha,q^*_\alpha) \Vdash^V_{\mathbb P \times \mathbb Q} \dot f(s_\alpha) = \beta$.  By the hypothesis about $\mathbb P$, let $X \subseteq \kappa_1$ be such that $X$ is closed under $f^\prime$, $|X| = \mu_1$ and $|X \cap \kappa_0| = \mu_0$.  Let $\gamma < \kappa_1$ be such that $X^{<\omega} \subseteq \{ s_\alpha : \alpha < \gamma \}$.  Next, take $H \subseteq \mathbb Q$ generic over $V[G]$ with $q^*_\gamma \in H$.  Then for all $\alpha < \gamma$, $\dot f^{G \times H}(s_\alpha) = f^\prime(s_\alpha)$, since $\{ (p^*_\alpha,q^*_\alpha) : \alpha < \gamma \} \subseteq G \times H$.  As $(p,q)$ was arbitrary, $\mathbb P \times \mathbb Q$ forces that there is $X \subseteq \kappa_1$ closed under $\dot f$ such that $|X \cap \kappa_0 | = \mu_0$.
\end{proof}

\section{Local Chang's Conjecture from subcompact cardinals}\label{sec: local cc}
In this section we will prove the consistency of certain instances of Chang's Conjecture relative to the existence of large cardinals at the level of supercompact cardinals.

We start with the concept of \emph{subcompactness}. This large cardinal notion was isolated by Jensen. We will use a generalization due to Brooke-Taylor and Friedman.
\begin{definition}\cite{BrookeTaylorFriedman2013}
Let $\kappa \leq \lambda$ be cardinals. $\kappa$ is $\lambda$-subcompact if for every $A\subseteq H(\lambda)$ there are $\bar\kappa, \bar\lambda < \kappa$, $\bar{A}\subseteq H(\bar\lambda)$ and an elementary embedding \[j\colon \langle H(\bar\lambda), \bar{A}, \in\rangle \to \langle H(\lambda), A, \in \rangle\]
with critical point $\bar\kappa$ and $j(\bar\kappa) = \kappa$.
\end{definition}
Following Neeman and Steel, we say that $\kappa$ is $(+\alpha)$-subcompact if it is $\kappa^{+\alpha}$-subcompact.

It follows immediately from a theorem of Magidor \cite[Lemma 1]{Magidor1971}, that $\kappa$ is $\lambda$-subcompact for all $\lambda$ iff $\kappa$ is supercompact. Nevertheless, subcompactness is level-by-level weaker that supercompactness.

Assuming \GCH, if $\kappa$ is $\kappa^{+\alpha + 1}$-supercompact cardinal and $\alpha < \kappa$, then $\kappa$ is $(+\alpha + 1)$-subcompact and the normal measure derived from the supercompact embedding concentrates on $(+\alpha + 1)$-subcompact cardinals.

On the other hand, if $\kappa$ is $(+\alpha + 2)$-subcompact then there are unboundedly many cardinals $\rho < \kappa$ such that $\rho$ is $\rho^{+\alpha + 1}$-supercompact.

\begin{theorem}\label{thm: cc at single point}
Assume \GCH. Let $\kappa$ be $(+2)$-subcompact. Then for every regular $\mu < \kappa$ there is $\rho < \kappa$ such that forcing with $\Col(\mu, \rho^{+})\times\Col(\rho^{++}, \kappa)$ forces $(\mu^{+3},\mu^{++})\chang(\mu^{+},\mu)$.
\end{theorem}
\begin{proof}
As a warm-up, let us show first an easier fact:
\begin{claim} There is $\rho < \kappa$ such that $(\kappa^{++}, \kappa^{+})\chang (\rho^{++}, \rho^{+})$
\end{claim}
\begin{proof}
Let $\kappa$ be a $(+2)$-subcompact cardinal. Assume, toward a contradiction, that for every $\rho < \kappa$, $(\kappa^{++}, \kappa^{+})\not\chang (\rho^{++}, \rho^+)$. In particular, for every $\rho < \kappa$ we can pick a function $f_\rho\colon (\kappa^{++})^{<\omega}\to\kappa^{+}$ such that for every $A\subseteq \kappa^{++}$, $|A|=\rho^{++}$, $|f_\rho\image A^{<\omega}| = \rho^{++}$.

Let us code the sequence $\langle f_\eta \colon \eta < \kappa\rangle$ as a subset of $H(\kappa^{++})$. By the $(+2)$-subcompactness of $\kappa$, there $\rho < \kappa$ and elementary embedding:
\[j\colon \langle H(\rho^{++}), \in, \langle g_\eta \colon \eta < \rho\rangle \rangle \to \langle H(\kappa^{++}),\in,\langle f_\eta \colon \eta < \kappa\rangle\rangle\]
Let us look at $A = j\image \rho^{++}$. By the definition of $f_\rho$, $f_\rho \image A^{<\omega}$ has cardinality $\rho^{++}$. In particular, there is a sequence of finite subsets of $A$, $\langle a_\xi \colon \xi < \rho^{++}\rangle$, such that $f_\rho (a_\xi) < f_\rho(a_\zeta) < \kappa^{+}$ for all $\xi < \zeta < \rho^{++}$. Since $a_\xi$ is a finite subset of $j\image \rho^{++}$, $a_\xi = j(b_\xi)$, where $b_\xi$ is a finite subset of $\rho^{++}$. For every pair $\xi < \zeta$, by elementarity, there is some $\eta < \rho$ such that $g_\eta (b_\xi) < g_\eta (b_\zeta) < \rho^{+}$.

Let us define  for $\xi < \zeta < \rho^{++}$, $c(\xi, \zeta) = \min \{\eta < \rho \colon g_\eta (b_\xi) < g_\eta(b_\zeta)
\}$. $c$ is a coloring of the pairs of ordinals below $\rho^{++}$. By \GCH, $2^\rho = \rho^+$. By the \Erdos-Rado Theorem, there is a homogeneous subset of $\rho^{++}$ with order type $\rho^{+} + 1$, $H$. Let $\eta$ be its color.

Let us look at the sequence $\langle g_\eta (b_\xi) \colon \xi \in H\rangle$. This is an increasing sequence of length $\rho^+ + 1$ of ordinals below $\rho^+$ - a contradiction.
\end{proof}
Let us now return to the proof of the theorem, which is very similar to the proof of the claim.

Let $\mathbb{L}_\rho = \Col(\mu, \rho^+)\times \Col(\rho^{++}, \kappa)$.

Assume, towards a contradiction, that there is no such $\rho$, i.e. for every $\rho < \kappa$ there is a $\mathbb{L}_\rho$-name, $\dot{f}_\rho$ of a function from $(\kappa^{++})^{<\omega}$ to $\kappa^{+}$ (in the sense of $V$) such that for every subset of cardinality $(\rho^{++})^{V}$, $A$, we have $\Vdash |\dot{f}_\rho\image (A^{{<}\omega})| = (\rho^{++})^{V}$. The sequence $\langle \dot{f}_\rho \colon \rho < \kappa\rangle$ can be coded as a subset of $H(\kappa^{++})$.

Using the $(+2)$-subcompactness, there is $j$ and $\rho$ such that:
\[j\colon \langle H(\rho^{++}), \in, \langle \dot{g}_\eta \colon \eta < \rho\rangle \rangle \to \langle H(\kappa^{++}),\in,\langle \dot{f}_\rho \colon \rho < \kappa\rangle\rangle\]
is elementary. As before, let us look at $A = j\image \rho^{++}$.

By the assumption, $\Vdash_{\mathbb{L}_\rho} |\dot{f}_\rho\image A^{<\omega}| = (\rho^{++})^{V}$.

Let $\langle \dot{x}_\alpha \colon \alpha < \rho^{++}\rangle$ be a sequence of $\mathbb{L}_\rho$-names, such that for every $\alpha < \beta$, $\Vdash \dot{f}_\rho(\dot{x}_\alpha) < \dot{f}_\rho(\dot{x}_\beta)$. Let us pick, by induction, conditions \[p_\alpha = \langle q_\alpha, r_\alpha\rangle \in \mathbb{L}_\rho = \Col(\mu, \rho^+)\times \Col(\rho^{++}, \kappa).\] For every $\alpha$, we find $a_\alpha \in A^{{<}\omega}$ and pick $p_\alpha$ such that $p_\alpha \Vdash \dot{x}_\alpha = \check{a}_\alpha$. Moreover, using the $\rho^{++}$-closure of $\Col(\rho^{++}, \kappa)$, we may pick the sequence $\langle r_\alpha \colon \alpha < \rho^{++}\rangle$ to be decreasing. $|\Col(\mu, \rho^{+})| = \rho^{+}$ and therefore, there is an unbounded subset $J \subseteq \rho^{++}$ such that for every $\alpha \in J$, $q_\alpha = q_\star$, for some fixed $q_\star \in \Col(\mu, \rho^{+})$. By re-arranging the sequence and omitting all elements outside of $J$, we may assume that $J = \rho^{++}$.

To conclude, we can find a sequence of conditions $\langle p_\alpha \colon \alpha < \rho^{++}\rangle$ and a sequence of finite subsets of $A$, $\langle a_\alpha\colon \alpha < \rho^{++}\rangle$, such that:
\begin{enumerate}
\item For all $\alpha$, $p_\alpha = \langle q_\star, r_\alpha\rangle$, where $q_\star\in\Col(\mu, \rho^{+})$ is fixed and $r_\alpha$ is decreasing.
\item For all $\alpha < \beta < \rho^{++}$, $p_\beta \Vdash \dot{f}_\rho (\check{a}_\alpha) < \dot{f}_\rho (\check{a}_\beta)$.
\end{enumerate}
Reflecting downward for every pair $\xi < \zeta$ separately and using the fact that $a_\xi = j(b_\xi)$ for some $b_\xi \in (\rho^{++})^{<\omega}$, we get that there is some $\eta < \rho$ and some condition $s$ in $\Col(\mu,\eta^+)\times\Col(\eta^{++}, \rho)$ such that $s\Vdash \dot{g}_\eta(\check{b}_\xi) < \dot{g}_\eta(\check{b}_\zeta) < (\check{\rho}^{+})^{V}$.

Let us define a coloring, $c\colon [\rho^{++}]^2\to \rho \times V_\rho$ that assigns for each pair $\xi < \zeta < \rho^{++}$ such pair $(\eta, s) \in \rho \times V_\rho$ as above. Since $\rho$ is measurable, and in particular inaccessible, this coloring obtains only $\rho$ many colors. Therefore, by \Erdos-Rado, we have a homogeneous set $H$ of order type $\rho^{+} + 1$. Let $(\eta, s)$ be its color.

For every $\xi < \zeta$ in $H$, $s\Vdash \dot{g}_\eta(\check{b}_\xi) < \dot{g}_\eta(\check{b}_\zeta)$. We conclude that in the generic extension relative to the forcing $\Col(\mu,\eta^{+})\times\Col(\eta^{++}, \rho)$, there is a subset of $\rho^{+}$ of order type $\rho^{+} + 1$, which is impossible.
\end{proof}
A similar method can be used in order to get instances of Chang's Conjecture with larger gaps between the cardinals, starting from stronger large cardinal assumptions:
\begin{theorem}
Assume \GCH, and let $\alpha \geq 2$ be an ordinal. If there is $\kappa > \alpha$ such that $\kappa$ is $(+\alpha)$-subcompact cardinal then there is $\rho < \kappa$ such that $\langle \kappa^{+i} \rangle_{1 \leq i \leq \alpha} \chang \langle \rho^{+i} \rangle_{1 \leq i \leq \alpha}$. Moreover, if $\alpha$ is a successor ordinal, we may find $\rho < \kappa$ such that
\[\Vdash_{\Col(\rho^{+\alpha},\kappa)} \langle \rho^{+\alpha+i} \rangle_{1 \leq i \leq \alpha} \chang \langle \rho^{+i} \rangle_{1 \leq i \leq \alpha}.\]
\end{theorem}

\begin{proof}
We will only sketch the case when we collapse cardinals.  Towards a contradiction, suppose $\kappa$ is $(+\alpha)$-subcompact and there is no such $\rho < \kappa$.  Let $\langle \dot f_\eta : \eta < \kappa \rangle$ be a sequence such that $\dot f_\eta$ is a $\Col(\eta^{+\alpha},\kappa)$-name for a counterexample.  Let $\rho < \kappa$ be such that there is an elementary embedding
\[j\colon \langle H(\rho^{+\alpha}), \in, \langle \dot{g}_\eta \colon \eta < \rho\rangle \rangle \to \langle H(\kappa^{+\alpha}),\in,\langle \dot{f}_\eta \colon \eta < \kappa\rangle\rangle.\]
Let $A = j \image \rho^{+\alpha}$, and consider a $\Col(\rho^{+\alpha},\kappa)$-name for $\dot f_\rho \image A$.  We may assume that $\Vdash A \subseteq \dot f_\rho \image A$, so that $\Vdash |\dot f_\rho \image A \cap \rho^{+\alpha+i}| \geq \rho^{+i}$ for $1 \leq i \leq \alpha$.  It suffices to prove that for all successor $\beta < \alpha$, $\Vdash |\dot f_\rho \image A \cap \rho^{+\alpha+\beta}| = \rho^{+\beta}$, since the limit cases follow by continuity.  For such $\beta$, the argument proceeds exactly as before.
\end{proof}

\begin{remark}
Suppose that in the above, $\alpha$ is the successor of a limit ordinal $\lambda$.  Using the lemmas of the previous section, we can then force with $\Col(\mu,\rho^{+\lambda})$, where $\mu \leq \cf(\lambda)$, to obtain $(\mu^{+\lambda+ 1}, \mu^{+\lambda}) \chang (\mu^+,\mu)$.  This gives an alternate proof of consistency results from \cite{HayutMagidorMalitz}.
\end{remark}

In the next theorem, we will get the consistency of instances of Chang's Conjecture where the source is double successor of singular cardinal, e.g.\ $(\aleph_{\omega+2}, \aleph_{\omega+1})\chang(\aleph_{n+1}, \aleph_n)$. In order to achieve this, we need to start with slightly stronger assumption.

\begin{lemma}\label{lem: cc between successor points}
Assume $\GCH$. Let $\kappa$ be $\kappa^{++}$-supercomapct cardinal and let $\mathcal{U}$ be a normal measure on $\kappa$, derived from a supercompact embedding.
\begin{enumerate}
\item There is a set of measure one $A$ such that for all $\rho\in A$, $(\kappa^{++}, \kappa^{+})\chang(\rho^{++},\rho^{+})$. Moreover, there is a set of measure one $A^\prime$ such that for every $\rho\in A^\prime$ and every forcing notion of cardinality $\leq \rho^+$, $\mathbb{Q}$, $\mathbb{Q}\times\Col(\rho^{++}, \kappa)$ forces $(\kappa^{++}, \kappa^{+})\chang(\rho^{++},|\rho^{+}|)$.
\item For every $\rho \in A$, there is a set $B_\rho\in\mathcal{U}$ such that for every $\eta\in B_\rho$, $(\eta^{++}, \eta^{+})\chang(\rho^{++},\rho^{+})$. Moreover, for every $\rho \in A^\prime$, there $B_\rho^\prime\in\mathcal{U}$ such that for every $\eta\in B^\prime_\rho$ and every forcing notion $\mathbb{Q}$, $|\mathbb{Q}| \leq \rho^{+}$, $\mathbb{Q} \times \Col(\rho^{++}, \eta)$ forces $(\eta^{++}, \eta^{+})\chang(\rho^{++},|\rho^{+}|)$.
\item There is $C\in\mathcal{U}$ such that for every $\zeta < \xi$ in $C$, $(\xi^{++}, \xi^{+})\chang(\zeta^{++},\zeta^{+})$ and moreover for every forcing notion $\mathbb{Q}$ or cardinality $\leq \zeta^+$, $\mathbb{Q} \times \Col(\zeta^{++}, \xi)$ forces $(\xi^{++}, \xi^{+})\chang(\zeta^{++},|\zeta^{+}|)$.
\end{enumerate}
\end{lemma}
\begin{proof}
Let us show the first assertion. Assume otherwise, and let us pick $\dot{f}_\rho, \mathbb{Q}_\rho$ witnessing it. Namely:
\begin{enumerate}
\item $\mathbb{Q}_\rho$ is a forcing notion of cardinality $\leq \rho^+$.
\item $\dot{f}_\rho$ is a name for a function from $(\kappa^{++})^{<\omega}$ into $\kappa^+$ in the forcing $\mathbb{Q}_\rho\times\Col(\rho^{++},\kappa)$
\item For every set of cardinality $\rho^{++}$, $X\subseteq \kappa^{++}$
\[\Vdash_{\mathbb{Q}_\rho\times\Col(\rho^{++},\kappa)} |\dot{f}_\rho\image \check{X}^{<\omega}| = \rho^{++}\]
\end{enumerate}
Let us assume, without loss of generality, that $\mathbb{Q}_\rho$ is a partial order on $\rho^+$.

Let $j\colon V\to M$ be a $\kappa^{++}$-supercompact embedding such that \[\mathcal{U} = \{Y\subseteq \kappa \colon \kappa\in j(Y)\}.\] Let us work in $M$. There, by elementarity, the forcing $j(\mathbb{Q})_\kappa\times \Col(\kappa^{++}, j(\kappa))$ adds a function $j(\dot{f})_\kappa\colon j(\kappa^{++})^{<\omega}\to j(\kappa^{+})$ such that for every $X\subseteq j(\kappa^{++})$ of cardinality $(\kappa^{++})^M = \kappa^{++}$
\[\Vdash_{j(\mathbb{Q})_\kappa\times\Col(\kappa^{++},j(\kappa))} |j(\dot{f})_\kappa\image \check{X}^{<\omega}| = \check{\kappa}^{++}\]
Take $X = j\image\kappa^{++}$. Let us denote $\mathbb{L} = j(\mathbb{Q})_\kappa\times\Col(\kappa^{++},j(\kappa))$. By the same arguments of Theorem~\ref{thm: cc at single point}, we can find a sequence of conditions $p_\alpha\in \mathbb{L}$ and $a_\alpha \subseteq \kappa^{++}$, finite, such that for every $\alpha < \beta$, $p_\beta \Vdash j(\dot{f})_\kappa(\widecheck{j(a_\alpha)}) < j(\dot{f})_\kappa(\widecheck{j(a_\beta)})$. Reflecting this fact back to $V$, we obtain that for every $\alpha < \beta < \kappa^{++}$ there is an ordinal $\rho < \kappa$ and a condition $r\in\mathbb{Q}_\rho \times \Col(\rho^{++}, \kappa)$ such that \[r\Vdash_{\mathbb{Q}_\rho \times \Col(\rho^{++},\kappa)} \dot f_\rho(\check a_\alpha) < \dot f_\rho(\check a_\beta) < \check \kappa^+.\]
This defines a coloring of pairs of elements in $\kappa^{++}$ with $\kappa$ many colors. By \Erdos-Rado, there is a homogeneous set of order type $\kappa^{+}+1$, which is impossible.

The second statement follows from the reflection properties of the supercompact cardinal. For the third statement take $C = A \cap \triangle_{\rho\in A} B_\rho$.
\end{proof}
\begin{theorem}
Assume \GCH. Let $\kappa$ be $\kappa^{++}$-supercompact and let $\mu < \kappa$ be regular. There is a generic extension in which $\kappa = \mu^{+\omega}$ and \[(\kappa^{++}, \kappa^+) \chang (\mu^+, \mu).\]
\end{theorem}
\begin{proof}
Let us consider the Prikry type forcing with collapses relative to a normal measure $\mathcal{U}$ on $\kappa$ which is a projection of $\kappa^{++}$-supercompact measure.
Let us describe explicitly the conditions in the forcing notion.

Let $j_\mathcal{U}\colon V\to M$ be the ultrapower embedding. Let us consider the forcing notion $\Col^M((\kappa^{++})^M, j(\kappa))$. This forcing is $\kappa^{+}$-closed (in $V$) and by standard counting arguments, it has only $|\mathcal{P}^M(j(\kappa))|^V = \kappa^{+}$ dense subsets in $M$. Therefore, there is an $M$-generic filter for $\Col^M((\kappa^{++})^M, j(\kappa))$ in $V$, $\mathcal{K}$.

Let us define the conditions for the forcing, $\mathbb{P}$.
\[p = \langle \alpha_0, f_0, \dots, \alpha_{n-1}, f_{n-1}, A, F \rangle\]
is a condition in $\mathbb{P}$ iff:
\begin{enumerate}
\item $\mu < \alpha_0 < \dots < \alpha_{n-1} < \kappa$.
\item $f_i\in\Col(\alpha_{i}^{++}, \alpha_{i+1})$ for $i < n - 1$, and $f_{n- 1} \in \Col(\alpha_{n - 1}^{++}, \kappa)$.
\item $A\in \mathcal{U}$, and $\min A > \alpha_{n - 1} + \rank f_{n - 1}$.
\item $F\colon A\to V$, for every $\alpha\in A$, $F(\alpha)\in\Col(\alpha^{++},\kappa)$ and $[F]_{\mathcal{U}}\in \mathcal{K}$.
\end{enumerate}
Let $p, q\in \mathbb{P}$.
\[p = \langle \alpha_0, f_0, \dots, \alpha_{n-1}, f_{n-1}, A, F\rangle\]
\[q = \langle \beta_0, g_0, \dots, \beta_{m-1}, g_{m-1}, B, G\rangle\]

$p\leq q$ iff:
\begin{enumerate}
\item $n \geq m$
\item For every $i < m$, $\alpha_i = \beta_i$.
\item For every $i < m$, $f_i \supseteq g_i$.
\item If $m \leq i <n$, then $\alpha_i\in B$ and $f_i \supseteq G(\alpha_i)$.
\item $A \subseteq B$ and for all $\alpha \in A$, $F(\alpha) \supseteq G(\alpha)$.
\end{enumerate}
We say that $p\leq^\star q$ if $p\leq q$ and they have the same length.
\begin{claim}
$\mathbb{P}$ satisfies the Prikry property. Namely, for every statement $\Phi$ in the forcing language and condition $p\in\mathbb{P}$, there is $q\leq^\star p$ that decides the truth value of $\Phi$.
\end{claim}
This is folklore, and a complete proof can be found in \cite{Faubion2012}. We remark that there are some minor differences between the presentation there and our presentation. As is standard, if $p = \langle \alpha_0, f_0, \dots, \alpha_{n}, f_n, A, F\rangle$, then $\mathbb P \restriction p \cong \Col(\alpha_{0}^{++},\alpha_1) \times \dots \times \Col(\alpha_{n-1}^{++},\alpha_{n}) \times \mathbb Q$, where $\mathbb Q$ adds no subsets of $\alpha_n^+$.

Since conditions with the same stem are compatible, $\mathbb P$ is $\kappa^+$-c.c.  Moreover, every ${<}\kappa$-sized set of conditions sharing a common stem has a lower bound with the same stem.
From the Prikry property and the chain condition, we conclude that every cardinal greater than or equal to $\kappa$ is preserved and that $\kappa$ becomes $\alpha_0^{+\omega}$.

Let $$X_0 = \{ \alpha < \kappa : \forall g \exists Y \exists F \langle \alpha, g ,Y,F \rangle \Vdash (\kappa^{++},\kappa^+) \chang (\alpha^{++},\alpha^+) \},$$
$$X_1 = \{ \alpha < \kappa : \exists g \exists Y \exists F \langle \alpha, g,Y,F \rangle \Vdash (\kappa^{++},\kappa^+) \not\chang (\alpha^{++},\alpha^+) \}.$$
By the Prikry property, either $X_0$ or $X_1$ is in $\mathcal U$. Towards a contradiction, suppose $X_1 \in \mathcal U$.  Let $\dot{f}\colon(\kappa^{++})^{<\omega}\to\kappa^{+}$ be a name for a function such that if $p_\alpha$ witnesses $\alpha \in X_1$, then $p_\alpha$ forces that for all $A \subseteq \kappa^{++}$ of size $\alpha^{++}$, $| \dot f \image A^{<\omega} | = \alpha^{++}$.

Let $j\colon V\to M$ be a $\kappa^{++}$-supercompact embedding such that the normal measure on $\kappa$ that it defines is the one that we use in the Prikry forcing.  Let $A = j \image \kappa^{++}$.
By hypothesis, there is a condition $p = j(p)_\kappa = \langle \kappa, g, Y,F \rangle \in j(\mathbb P)$ such that $p \Vdash^M_{j(\mathbb P)} | j(\dot f) \image A^{<\omega} | = \kappa^{++}$.  There is a sequence of names $\langle \dot b_\alpha : \alpha < \kappa^{++} \rangle \subseteq A^{<\omega}$ such that $p$ forces $\langle \dot f(b_\alpha) : \alpha < \kappa^{++} \rangle$ is an increasing sequence of ordinals below $j(\kappa^+)$.  The sequence of $\dot b_\alpha$'s is forced to be $j \image \vec a$, where $\vec a$ is a sequence contained in $(\kappa^{++})^{<\omega}$.  Let $q = \langle \kappa,g_0', \gamma,g_1',Y',F' \rangle \leq p$.  By the Prikry property, $\vec a$ is added by the factor $\Col(\kappa^{++},\gamma)$, and there is an extension $q'$ of $q$ of the same length that decides on a $\Col(\kappa^{++},\gamma)$-name $\dot{\tau}$ for $\vec a$.

Let $\langle r_\alpha : \alpha < \kappa^{++} \rangle$ be a descending sequence of conditions in $\Col(\kappa^{++},\gamma)$ such that $r_0 \leq g_0'$, and $r_\alpha$ decides a value $a_\alpha$ the $\alpha$-th value in the vector $\dot\tau$, $\dot{\tau}(\alpha)$. Thus we have that for each $\alpha < \beta < \kappa^{++}$,
$$\langle \kappa,r_\beta, \gamma, g_1', Y',F' \rangle \Vdash j(\dot f)(\widecheck{j( a_\alpha)}) < j(\dot f)(\widecheck{j(a_\beta)}) < j(\kappa^+).$$
Reflecting this downwards to $V$, we have the following coloring problem: For every pair of ordinals, $\alpha < \beta < \kappa^{++}$, there is a condition $p_{\alpha,\beta}$ that forces $\dot f(a_\alpha) < \dot f(a_\beta) < \kappa^+$. We want to show that this situation is impossible.

There are $2^\kappa$ many conditions in $\mathbb{P}$ and therefore we can think of the above coloring as a coloring from $[\kappa^{++}]^2$ to $2^\kappa$. By Remark~\ref{remark: erdos-rado for 2k colors} we obtain an ordinal $\delta <\kappa^{++}$ and a sequence of ordinals $\{\alpha_i \colon i < \kappa^{+}\}$, cofinal in $\delta$ such that $p_{\alpha_\xi, \alpha_\zeta} = p_{\alpha_\xi, \delta}$ for all $\xi < \kappa^{+}$.
Let us denote by $q_i = p_{\alpha_i, \delta}$. Let us claim that there is a condition $r$ that forces that the set $\{i < \kappa^{+} \colon q_i \in G\}$ (where $G$ is the generic filter) is unbounded. Otherwise, using the chain condition of the forcing, we can find an ordinal $\beta<\kappa^+$ such that $\Vdash \{i < \kappa^{+} \colon q_i \in G\} \subseteq \check\beta$. This is absurd, since $q_{\beta}$ forces the opposite statement.

Let $G$ be a generic filter containing $r$. Let $I = \{i < \kappa^{+} \colon q_i \in G\}$. Since $\kappa^{+}$ is regular in $V[G]$, $\otp I = \kappa^{+}$. For every $\xi, \zeta \in I$, $\xi < \zeta$, $p_{\alpha_\xi, \alpha_\zeta} = p_{\alpha_\xi, \delta} = q_\xi \in G$. Therefore, $V[G]\models f(a_\xi) < f(a_\zeta) < f(a_\delta) < \kappa^+$.
So, in $V[G]$, there is a sequence of ordertype $\kappa^{+} + 1$ of ordinals below $\kappa^{+}$ -- a contradiction.

Therefore, $X_0 \in \mathcal U$, and any condition of the form $\langle X_0,F \rangle$ forces $(\kappa^{++},\kappa^+) \chang (\alpha_0^{++},\alpha_0^+)$, where $\alpha_0$ is the first member of the Prikry sequence.  If we then force with $\Col(\mu,\alpha_0^+)$, we get the desired conclusion.
\end{proof}

\begin{corollary}
It is consistent relative to a $(+2)$-supercompact cardinal that for all even ordinals $\beta < \alpha \leq \omega$, $(\aleph_{\alpha+2},\aleph_{\alpha+1}) \chang (\aleph_{\beta+2},\aleph_{\beta+1})$.
\end{corollary}

\begin{proof}
If $p = \langle \alpha_0, f_0, \dots, \alpha_{n}, f_n, A, F\rangle \in \mathbb P$ as above, then
$$\mathbb P \restriction p \cong \Col(\alpha_0^{++},\alpha_1) \times \dots \times \Col(\alpha_{n-1}^{++},\alpha_{n}) \times \mathbb P \restriction \langle \alpha_n,f_n,A,F \rangle.$$
If $G \subseteq \mathbb P$ is generic, then for some $n_0 < \omega$, we have $(\kappa^{++},\kappa^+) \chang (\alpha_m^{++},\alpha_m^+)$ for all $m \geq n_0$.  By Lemma \ref{lem: cc between successor points}, there is some $n_1 \geq n_0$ such that $(\alpha_{m+1}^{++},\alpha_{m+1}^{+}) \chang (\alpha_m^{++},\alpha_m^{+})$ for all $m \geq n_1$.  These relations continue to hold after forcing with $\Col(\omega,\alpha_{n_1})$.  By the facts mentioned after Remark \ref{remark: erdos-rado for 2k colors}, the desired statement holds in this extension.
\end{proof}

The learned reader may perceive how to use Radin forcing to extend the above result to obtain a model of ZFC in which for \emph{all} even ordinals $\beta < \alpha$, $(\aleph_{\alpha+2},\aleph_{\alpha+1}) \chang (\aleph_{\beta+2},\aleph_{\beta+1})$.  Instead of pursuing this line, we will now work towards showing the consistency of a ``denser'' global Chang's Conjecture using much stronger large cardinal assumptions, answering a question of Foreman.

\section{Global Chang's Conjecture}\label{sec: global cc}
In this section we obtain a model in which $(\kappa^{+}, \kappa) \chang (\mu^{+}, \mu)$
holds for every cardinal $\mu$ and every successor cardinal $\kappa$, starting from a huge cardinal.
\subsection{Getting Chang's Conjecture between many pairs of regular cardinals}
Towards the goal of the section, we start with a simpler task:
\[(\nu^{+}, \nu) \chang (\mu^+, \mu)\]
for all regular cardinals $\mu < \nu$.  This answers Question 7 of Foreman in \cite{Foreman2010ideals}.  Foreman asked whether a weaker statement is consistent, where we assume the larger $\nu$ is a successor cardinal.  In our model, we retain many large cardinals.  We will then force further to obtain a model in which the smaller cardinal $\mu$ can be singular.

\begin{lemma}
\label{easton projection}
Suppose $\mu<\kappa \leq \lambda < \delta$ are regular, $\kappa$ and $\delta$ are Mahlo, and
\begin{enumerate}
\item $\Vdash_{\E(\mu,\kappa)}$ ``$\dot{\mathbb Q}$ is $\mu$-closed, of size $\leq \lambda$, and preserves the regularity of $\lambda$.''
\item $\Vdash_{\E(\mu,\kappa) * \dot{\mathbb Q}}$ ``$\dot{\mathbb R}$ is $\mu$-closed and of size ${<}\delta$.''
\end{enumerate}
Then there is a projection $\pi : \E(\mu,\delta) \to (\E(\mu,\kappa) * \dot{\mathbb Q}) * (\dot{\mathbb R} \times \dot{\E}(\lambda,\delta))$ such that for all $p$, the first coordinate of $\pi(p)$ is $p{\restriction}\kappa$.
\end{lemma}
\begin{proof}
For brevity, let $\mathbb P = \E(\mu,\kappa) * \dot{\mathbb Q}$.  For every inaccessible cardinal $\alpha \in (\lambda,\delta)$, $T(\mathbb P, \dot{\Col}(\lambda,\alpha))$ is $\lambda$-closed and has size $\alpha$. By Lemma \ref{folk}, since $\mathbb P * \dot{\Col}(\lambda,\alpha)$ collapses $\alpha$ to have size $\lambda$, and since $|\mathbb{P}| < \alpha$, $T(\mathbb P, \dot{\Col}(\lambda,\alpha))$ collapses $\alpha$ to $\lambda$ as well and therefore $T(\mathbb P, \dot{\Col}(\lambda,\alpha))$ is forcing-equivalent to $\Col(\lambda,\alpha)$. Thus,
\[ \mathbb P \times \E(\lambda,\delta) \cong \mathbb P \times \prod_{\alpha \in (\lambda,\delta) \cap I}^E T(\mathbb P, \dot{\Col}(\lambda,\alpha)), \]
where $I$ is the class of inaccessible cardinals, and the superscript $E$ indicates that Easton supports are used.  We define a projection
\[ \pi_0 : \mathbb P \times \prod_{\alpha \in (\lambda,\delta) \cap I}^E T(\mathbb P, \dot{\Col}(\lambda,\alpha)) \to \mathbb P * \dot{\E}(\lambda,\delta) \]
as follows.  $\pi_0(p,q) = \langle p,\tau_q \rangle$, where $\tau_q$ is the canonical $\mathbb P$-name for a function with domain $\dom q$, and $\Vdash \forall \alpha,\, \dot \tau_q(\alpha) = \check q(\alpha)$.  To show $\pi_0$ is a projection,  suppose $\langle p_1,\dot q_1 \rangle \leq \langle p_0,\tau_{q_0} \rangle$.  Since $|\mathbb P| \leq \lambda$, there is an Easton set $X$ such that $\Vdash \dom(\dot q_1) \subseteq \check X$.  For each $\alpha \in X$, let $\sigma_\alpha$ be a $\mathbb P$-name such that $p_1 \Vdash \sigma_\alpha = \dot q_1(\check \alpha)$ and if $p \perp p_1$, $p \Vdash \sigma_\alpha= \check q_0(\check \alpha)$.  If $q_2 = \{ \langle \alpha,\sigma_\alpha \rangle : \alpha \in X \}$, then $\langle p_1,q_2 \rangle \leq \langle p_0,q_0 \rangle$, and $\langle p_1,\tau_{q_2} \rangle \leq \langle p_1,\dot q_1 \rangle$ because $p_1 \Vdash \tau_{q_2} = \dot q_1$.

Let $\rho < \delta$ be regular and such that $\Vdash_{\mathbb P} |\dot{\mathbb R}| \leq \rho$. Applying Lemma \ref{folk} coordinate-wise, we have the following sequence of projections:
\begin{align*}
\E(\mu,\delta)	&\cong \E(\mu,\kappa) \times \E(\mu,\delta) \restriction [\kappa,\delta) \\
			&\cong \E(\mu,\kappa) \times \Col(\mu,\lambda) \times \Col(\mu,\rho) \times \E(\mu,\delta) \restriction (\lambda,\delta) \\
			&\to (\E(\mu,\kappa) * \dot{\mathbb Q} * \dot{\mathbb R}) \times \E(\lambda,\delta) \\
			&\cong (\mathbb P *  \dot{\mathbb R}) \times \prod_{\alpha \in (\lambda,\delta) \cap I}^E T(\mathbb P, \dot{\Col}(\lambda,\alpha)) \\
			&\to \mathbb P * (\dot{\mathbb R} \times  \dot{\E}(\lambda,\delta)),
\end{align*}
as desired.
\end{proof}

The next lemma answers a question of Shioya \cite{Shioya2011}, who asked if $\kappa$ is a huge cardinal with target $\delta$, and $\mu< \kappa$ is regular, does $\E(\mu,\kappa) * \dot \E(\kappa,\delta)$ force $(\mu^{++},\mu^{+})\chang(\mu^+,\mu)$?  He noted in the same paper that if we allow more distance between the cardinals, then the answer is yes: for example $\E(\mu,\kappa) * \dot \E(\kappa^+,\delta)$ forces $(\mu^{+3},\mu^{++})\chang(\mu^+,\mu)$.  The next lemma covers these cases and many others.  The main issue is to understand the behavior of a potential master condition.  The argument shows that the use of posets like the Silver collapse, which is designed to get master conditions under control (see \cite{Foreman2010ideals}), is not actually needed in this context.

\begin{lemma}
\label{lemma: easton collapse give cc}
Suppose $\kappa$ is a huge cardinal, $j : V \to M$ is a huge embedding with critical point $\kappa$, $j(\kappa) = \delta$, and $\mu,\lambda$ are regular cardinals with $\mu<\kappa \leq \lambda < \delta$.  Suppose also:
\begin{enumerate}
\item $\Vdash_{\E(\mu,\kappa)}$ ``$\dot{\mathbb Q}$ is $\check\kappa$-directed-closed, of size $\leq \check\lambda$, and preserves the regularity of $\check\lambda$.''
\item $\Vdash_{\E(\mu,\kappa) * \dot{\mathbb Q}}$ ``$\dot{\mathbb R}$ is $\check\kappa$-directed-closed, of size ${<}\check\delta$.''
\end{enumerate}
Then it is forced by $(\E(\mu,\kappa) * \dot{\mathbb Q}) * (\dot{\mathbb R} \times \dot{\E}(\lambda,\delta))$ that $(\lambda^+,|\lambda|) \chang (\mu^+,\mu)$.
\end{lemma}
\begin{proof}
Let $G * g * (h \times H)$ be $(\E(\mu,\kappa) * \dot{\mathbb Q}) * (\dot{\mathbb R} \times \dot{\E}(\lambda,\delta))$-generic. By Lemma \ref{easton projection} there is a further forcing yielding an $\E(\mu,\delta)$-generic $\hat G$ with $G * g * (h \times H) \in V[\hat G]$.  The embedding can be extended to $\bar{j} : V[G] \to M[\hat G]$.

Since $M[\hat G] \models |g*h| < \check\delta$, and $\bar{j}(\mathbb Q * \mathbb R)$ is $\delta$-directed-closed, there is a master condition $(q,r) \in \bar{j}(\mathbb Q * \mathbb R)$ below $\bar{j} \image (g * h)$.  Forcing below this, we get an extended embedding $\hat{j} : V[G*g*h] \to M[\hat G * \hat g * \hat h]$.

In $M[\hat{G} * \hat g]$, for any $\alpha < \delta$, $\hat{j} \image H {\restriction} \alpha$ is a directed subset of $E(j(\lambda),j(\alpha))^{M[\hat{G} * \hat g]}$ of size $<\delta$.  Hence we define $m_\alpha = \inf \hat{j} \image H {\restriction} \alpha$.

Note that the restriction maps are continuous in the sense that for any ordinals $\alpha < \beta < \gamma$ and any $X \subseteq \E(\alpha,\gamma)$ with a lower bound, $(\inf X) {\restriction} \beta = \inf \{ p {\restriction} \beta : p \in X \}$.  Since  $H {\restriction} \alpha = \{ p {\restriction} \alpha : p \in H {\restriction} \beta \}$ for any $\alpha < \beta < \delta$, we have:

\begin{align*}
m_\beta {\restriction} j(\alpha) & =(\inf \{ \hat{j}(p) : p \in H {\restriction} \beta \}) \restriction j(\alpha) = \inf \{ \hat{j}(p) {\restriction} j(\alpha) : p \in H {\restriction} \beta \} \\
& = \inf \{ \hat{j}(p {\restriction} \alpha) : p \in H {\restriction} \beta \}  = \inf \{ \hat{j}(p) : p \in H {\restriction} \alpha \} = m_\alpha.
\end{align*}

In $M[\hat{G} * \hat g]$, let $m =  \bigcup_{\alpha < \delta} m_\alpha$. To show that $m \in E(j(\lambda),j(\delta))^{M[\hat G * \hat g]}$, let $\gamma = \sup j \image \delta < j(\delta)$.  First note $M[\hat G * \hat g]$ thinks that $\dom m$ is an Easton subset of $\gamma$.  This is because $\dom m = \bigcup_{p \in H} \dom j(p)$, $M[\hat G] \models |H| = \delta$, and for all $p \in H$, $M[\hat G] \models$ ``$\dom j(p)$ is an Easton subset of $\gamma \setminus \delta$.''  Second, for every $\beta < \gamma$, $\dom m(\beta)$ is a bounded subset of $j(\lambda)$.  For $\beta < \gamma$, if $\alpha < \delta$ is such that $j(\alpha) > \beta$, then $m(\beta)$ is ``frozen'' by $m_\alpha$, as $m_\xi(\beta) = m_\alpha(\beta)$ for all $\xi \geq \alpha$.

Therefore if we take a generic $\hat{H} \subseteq E(j(\lambda),j(\delta))^{M[\hat{G} *\hat g]}$ over $V[\hat{G} * \hat g * \hat h]$ with $m \in \hat{H}$, then we get an extended elementary embedding $\tilde{j} : V[G * g * (h \times H)] \to M[\hat G * \hat g * (\hat h \times \hat H)]$.  If $f : \delta^{{<}\omega} \to \delta$ is in $V[G * g * (h \times H)]$, then $j \image \delta = \tilde{j} \image \delta$ is closed under $j(f)$.  In $M[\hat G * \hat g * (\hat h \times \hat H)]$, $| j \image \delta \cap j(\lambda) | = |\lambda| = \mu$, and $j \image \delta$ has size $\delta = j(\kappa)$.  Thus by elementarity, there is some $X \subseteq \delta$ of size $\kappa$ closed under $f$ and such that $|X \cap \lambda | = \mu$ in $V[G * g * (h \times H)]$.
\end{proof}

\begin{lemma}
\label{below huge}
Let $\kappa$ be huge with witnessing embedding $j : V \to M$, and let $\mathcal{U}$ be the normal measure on $\kappa$ derived from $j$.
There is $A\in\mathcal{U}$ such that for all regular cardinals $\alpha < \beta \leq \gamma < \delta$ with $\beta, \delta \in A$, and for every notion of forcing of the form $(\E(\alpha,\beta) * \dot{\mathbb Q}) * (\dot{\mathbb R} \times \dot{\E}(\gamma,\delta))$, where
\begin{enumerate}
\item $\Vdash_{\E(\alpha,\beta)}$ ``$\dot{\mathbb Q}$ is $\beta$-directed-closed, of size $\leq \gamma$, and preserves the regularity of $\gamma$,''
\item $\Vdash_{\E(\alpha,\beta) * \dot{\mathbb Q}}$ ``$\dot{\mathbb R}$ is $\beta$-directed-closed and of size ${<}\delta$.''
\end{enumerate}
forces $(\gamma^+, |\gamma|)\chang (\alpha^+, \alpha)$.
\end{lemma}
\begin{proof}
Let $\varphi(\beta,\delta)$ stand for the assertion that whenever $\alpha < \beta \leq \gamma < \delta$ are regular and
 (1) and (2) hold of $\dot{\mathbb Q}$ and $\dot{\mathbb R}$ as above, then $(\E(\alpha,\beta) * \dot{\mathbb Q}) * (\dot{\mathbb R} \times \dot{\E}(\gamma,\delta))$ forces $(\gamma^+, |\gamma|)\chang(\alpha^+, \alpha)$.

By Lemma \ref{lemma: easton collapse give cc}, $V \models \varphi(\kappa,j(\kappa)).$  Since $M^{j(\kappa)} \subseteq M$, this holds in $M$ as well.  Reflecting this statement once, we find $A_0\in\mathcal{U}$ such that $\varphi(\beta,\kappa)$ holds for all $\beta \in A_0$.  Next, for all $\beta \in A_0$, we can reflect again to find a set $A_\beta \in \mathcal U$ such that $\varphi(\beta,\delta)$ holds for all $\delta \in A_\beta$.  Take $A = A_0 \cap \triangle_{\beta \in A_0} A_\beta$.
\end{proof}

Let $\kappa$ be a huge cardinal, with $A \subseteq \kappa$ as above, and without loss of generality assume $A$ contains only ${<}\kappa$-supercompact cardinals. Let $\langle \alpha_i : i < \kappa \rangle$ be the increasing enumeration of the closure of  $A \cup \{\omega\}$.  We define an Easton-support iteration $\langle \mathbb P_i, \dot{\mathbb Q}_j : i \leq \kappa, j < \kappa \rangle$ as follows:

If $\alpha_i$ is regular, $\Vdash_i \dot{\mathbb Q}_i = \dot{\E}(\alpha_i,\alpha_{i+1})$.  If $\alpha_i$ is singular, $\Vdash_i \dot{\mathbb Q}_i = \dot{\E}(\alpha_i^+,\alpha_{i+1})$.  Note that since each cardinal in $A$ is sufficiently supercompact, the cardinality assumptions of Lemma \ref{below huge} are satisfied at singular limits.  If $i<j$, then $\mathbb P_{j+1}$ forces $(\alpha_{j+1},\beta)\chang(\alpha_i^+,\alpha_i)$, where $\beta$ is the cardinal predecessor of $\alpha_{j+1}$ in $\mathbb P_{j+1}$.  By Lemma \ref{preserve closed}, this is preserved by the tail $\mathbb P_\kappa / \mathbb P_{j+1}$, which is $\alpha_{j+1}$-closed.

After forcing with $\mathbb P_\kappa$, we have:
\begin{claim}\label{cc between all successors of regulars}
$V_\kappa^{\mathbb{P}_\kappa} \models (\alpha^+,\alpha) \chang (\beta^+,\beta)$ for all pairs of regular cardinals $\beta < \alpha$.  Furthermore, this is preserved by $\Col(\gamma_0,\gamma_1)$, whenever $\beta \leq \gamma_0 \leq \gamma_1 \leq \alpha$ are regular (where $\beta < \alpha$).
\end{claim}
We do not exclude the cases in which $\beta = \gamma_0$ or $\alpha = \gamma_1$. Note that if both equations hold then in the generic extension $\beta^{+} = \alpha^{+}$ and $|\alpha| = \beta$. In this case the assertion holds trivially.

The stronger claim holds because if $\beta$ is the next regular cardinal $\geq |\mathbb P_i|$ and $\alpha$ is such for $\mathbb P_j$, then it is forced by $\mathbb P_i$ that $(\mathbb P_{j+1} / \mathbb P_i)\times \Col(\gamma_0,\gamma_1)$ has a form satisfying the hypotheses of Lemma \ref{below huge}. As $\Col(\gamma_0,\gamma_1)$ is $\alpha^+$-c.c.\ and $\mathbb P_\kappa / \mathbb P_{j+1}$ is $\alpha^+$-closed, Lemma \ref{preserve closed} implies this instance of Chang's Conjecture continues to hold in $V^{\mathbb P_\kappa \ast \Col(\gamma_0,\gamma_1)}$.

\begin{claim}\label{claim:preserving almost hugeness}
Assume that $\kappa$ is almost huge cardinal with a witnessing elementary embedding $j$ such that $\delta=j(\kappa)$ is Mahlo and $\kappa \in j(A)$. Assume also that $\sup j\image \delta = j(\delta)$. Then in the generic extension by $j(\mathbb{P}_\kappa)$, $j$ extends to an elementary embedding witnessing that $\kappa$ is almost huge.
\end{claim}
\begin{proof}
Let $\hat{\mathbb{P}} = j(\mathbb{P}_\kappa)$. Let us analyse the forcing notion $j(\hat{\mathbb{P}})$.

$\hat{\mathbb{P}}$ is defined as an Easton support iteration of Easton collapses, between the elements in the closure of the set $j(A)$. Note that since $M \cap V_{\delta} = V_{\delta}$, $V$ and $M$ compute this iteration in the same way. $j(\hat{\mathbb P})$ is an Easton support iteration in the model $M$, of length $j(\delta)$, between the points in the closure of $j^2(A)$. Note that \[j^2(A) \cap \delta = j(j(A) \cap \kappa) = j(A).\] Therefore, $j(\hat{\mathbb{P}}) = \hat{\mathbb{P}} \ast \dot{\mathbb{Q}}$ where  $\dot{\mathbb{Q}}$ is forced to be a $\delta$-closed forcing notion.

Let $G \subseteq \hat{\mathbb{P}}$ be a $V$-generic filter. In $V[G]$, we will define a filter $H \subseteq j(\hat{\mathbb{P}})$ generic over $M[G]$ such that for every $p\in G$, $j(p) \in G \ast H$.

We imitate the proof of Lemma~\ref{lemma: easton collapse give cc}. For every $\alpha < \delta$ inaccessible, let $m_\alpha = \bigcup_{p\in G \cap V_\alpha} j(p)$.
Since we apply $j$ on $\alpha$ many elements and $\alpha < \delta$, $m_\alpha \in M[G]$. Also, for every $\alpha < \beta$, $m_\beta \restriction j(\alpha) = m_\alpha$.

$\hat{\mathbb{P}}$ is $\delta$-c.c.\ and therefore $M$ models that $j(\hat{\mathbb{P}})$ is $j(\delta)$-c.c. As $\delta$ is inaccessible, $M[G]$ can compute an enumeration of the set of all maximal antichains of $j(\hat{\mathbb{P}})\restriction [\delta, j(\delta))$ in a sequence of length $j(\delta)$. $|j(\delta)|^V = \delta$, so $V[G]$ can enumerate those maximal antichains in a sequence of length $\delta$. Since all those antichains are bounded below $j(\delta)$, we may pick an enumeration $\langle \mathcal{A}_i \colon i < \delta\rangle \in V[G]$ in which $\mathcal{A}_\alpha$ is a maximal antichain in $j(\hat{\mathbb{P}} \restriction \alpha)$ (here we use the fact that $\sup j\image \delta = j(\delta)$).

Let us define a decreasing sequence of conditions $\langle q_i \colon i < \delta\rangle \subseteq \mathbb{Q}$. We require that $q_\alpha \leq m_\alpha$, $q_\alpha \in \mathcal{A}_\alpha$ and that the support of $q_\alpha$ is a subset of $j(\alpha)$. For every $\alpha < \delta$, the sequence $\langle q_\beta \colon \beta < \alpha\rangle$ is a member of $M$. By the $\delta$-closure of $\mathbb{Q}$, one can always pick a condition $q_\alpha$ stronger than all previous conditions. By the properties of $\mathbb{Q}$, it is clear that one can choose $q_\alpha$ to have support which is contained in the union of the supports of $q_\beta$, $\beta < \alpha$.

Let $H$ be the filter generated by $\{q_i \colon i < \delta\}$. By the construction of this sequence, $H$ meets every maximal antichain in $M$ of the forcing notion $j(\hat{\mathbb{P}})$. Therefore, it is $M[G]$-generic.

By Silver's criteria, $j\colon V\to M$ extends to an elementary embedding $\tilde{j}\colon V[G] \to M[G][H]$. Let us claim that $M[G][H]^{{<}\delta} \subseteq M[G][H]$. Indeed, let $\langle \dot{x}_i \colon i < \alpha\rangle$ be a sequence of names of elements of $M$ and assume that $\alpha < \delta$. Without loss of generality, we may assume that $\dot{x_i}$ is a name of an ordinal for every $i < \alpha$. By the chain condition of $\hat{\mathbb{P}}$, this sequence can be encoded as a set of ordinals of cardinality $<\delta$, and therefore belongs to $M$. Since $G\in M[G][H]$, we conclude that also its realization is in $M$, as needed.
\end{proof}

In fact, for our goals it is sufficient to note only that some of the supercompactness of $\kappa$ is preserved.

For the next section we need a stronger version of Lemma~\ref{lemma: easton collapse give cc} and Claim~\ref{cc between all successors of regulars}. We will need to know that a stronger type of reflection holds between pair of elements from $A$.

\begin{definition}[Magidor-Malitz Quantifiers]
Let $M$ be a model over the language $\mathcal{L}$. We enrich $\mathcal{L}$ with the quantifiers $Q^n$ with the following interpretation:
\[ M \models Q^n x_0, \dots x_{n-1} \varphi(x_0, \dots, x_{n-1}, p)\]
iff
\[\exists I \subseteq M,\, |I| = |M|,\, \forall a_0,\dots, a_{n-1}\in I,\, M \models \varphi(a_0, \dots, a_{n-1}, p)\]
\end{definition}

A set $I\subseteq M$ satisfying
\[\forall a_0,\dots, a_{n-1}\in I,\, M \models \varphi(a_0, \dots, a_{n-1}, p)\]
is called a $\varphi$-block.

The Magidor-Malitz quantifiers were defined by Menachem Magidor and Jerome Malitz in \cite{MagidorMalitz1977}. In this paper, they showed that under $\diamondsuit(\aleph_1)$ a certain compactness theorem holds for the language $\mathcal{L}(Q^{<\omega})$ - the first order logic extended by adding the Magidor-Malitz quantifiers.

We say that $A\prec_{Q^n} B$ if $A$ is $\mathcal{L}(Q^n)$-elementary substructure of $B$. We write $\mu\chang_{Q^n} \nu$ if for every model $B$ of cardinality $\mu$, there is a $Q^n$-elementary submodel $A$ of cardinality $\nu$.

\begin{lemma}
Suppose $\kappa$ is a huge cardinal, $j : V \to M$ is a huge embedding with critical point $\kappa$, $j(\kappa) = \delta$, and $\mu,\lambda$ are regular cardinals with $\mu<\kappa \leq \lambda < \delta$.  Suppose also:
\begin{enumerate}
\item $\Vdash_{\E(\mu,\kappa)}$ ``$\dot{\mathbb Q}$ is $\kappa$-directed-closed, of size $\leq \lambda$, and preserves the regularity of $\lambda$.''
\item $\Vdash_{\E(\mu,\kappa) * \dot{\mathbb Q}}$ ``$\dot{\mathbb R}$ is $\kappa$-directed-closed, of size ${<}\delta$.''
\end{enumerate}
Then it is forced by $(\E(\mu,\kappa) * \dot{\mathbb Q}) * (\dot{\mathbb R} \times \dot{\E}(\lambda,\delta))$ that $\lambda^+ \chang_{Q^{<\omega}} \mu^+$.
\end{lemma}
\begin{proof}
First, let us note that it is enough to deal with models $\mathcal{A}$ which are transitive elementary submodels of $H(\delta^{+})$. Indeed let $\mathcal{A}^\prime$ be an algebra on $\delta$. Clearly, $\mathcal{A}^\prime \in H(\delta^+)$. Let $\mathcal{A}$ be a transitive elementary submodel of $H(\delta^{+})$ of cardinality $\delta$ that contains $\mathcal{A}^\prime$ as an element. Assume that $\mathcal{B} \prec_{Q^n} \mathcal{A}$. Let us claim that $\mathcal{B}^\prime = \mathcal{B} \cap \mathcal{A}^\prime \prec_{Q^n} \mathcal{A}^\prime$. This is true, as any $Q^n$ statement in $\mathcal{A}^\prime$ is equivalent to a $Q^n$ statement in $\mathcal{A}$.

Let $\mathcal A$ be a transitive elementary structure of $H(\delta^+)$ of size $\delta$. We may assume that for every formula $\Phi$ of the form $Q^n x_0, \dots, x_{n-1} \varphi(x_0, \dots, x_{n-1}, p)$ there is function in the language of $\mathcal{A}$, $f_\Phi$ such that $f_\Phi\colon \mathcal{A}\to \mathcal{A}$ is either constant (if $\neg \Phi$) or one-to-one (if $\Phi$ holds), and if it is one-to-one, then \[\mathcal{A} \models \forall y_0, \dots, y_{n-1} \varphi(f_\Phi(y_0), \dots, f_\Phi(y_{n-1}), p).\]

Let $j : V[G * g * (h \times H)]  \to M[\hat G * \hat g * (\hat h \times \hat H)]$ be as in Lemma \ref{lemma: easton collapse give cc}, and for brevity denote the domain and codomain by $V'$ and $M'$ respectively.  We want to show that $j \image \mathcal A$ is a $Q^{{<}\omega}$-elementary substructure of $j(\mathcal A)$.

Let $\Phi$ be a formula of the form:
\[Q^n x_0, \dots, x_{n-1} \varphi(x_0, \dots, x_{n-1}, j(p)),\]
Let us assume, by induction, that every proper subformula of $\Phi$ is satisfied by $j(\mathcal{A})$ in $M'$ if and only if it is satisfied by $j\image \mathcal{A}$ in $M'$.

First, let us assume $M' \models ``j(\mathcal{A}) \models \Phi."$ By elementarity, \[V' \models `` \mathcal{A} \models Q^n x_0, \dots, x_{n-1} \varphi(x_0, \dots, x_{n-1}, p)."\] By the observation above, there is a function $f_\Phi$ witnessing this fact in $V'$ and clearly, $j(f_\Phi)$ is a one-to-one function on $j \image \mathcal{A}$ witnessing $j \image \mathcal{A} \models \Phi$.

On the other hand, assume that
\[M' \models `` j\image \mathcal{A} \models Q^n x_0, \dots, x_{n-1} \varphi(x_0, \dots, x_{n-1}, j(p))."\]
Let $I\subseteq j\image \mathcal{A}$ be a $\varphi$-block. We want to show that there is a corresponding $\varphi$-block (for the parameter $p$) also in $V'$.  Note that the forcing to obtain $\hat G * \hat g * (\hat h \times \hat H)$ from $G * g * (h \times H)$ is of the form $\mathbb{Q}_0\ast \mathbb{Q}_1$ where $\mathbb{Q}_0$ is a precaliber-$\delta$ forcing, and $\mathbb{Q}_1$ is a $\delta$-closed forcing. Let us denote the generic filter for $\mathbb{Q}_0$ by $K_0$ and the generic filter for $\mathbb{Q}_1$ by $K_1$.

In order to find the $\varphi$-block in $V'$, we will show that the existence of such $\varphi$-block in $V'[K_0][K_1]$ implies the existence of a corresponding $\varphi$-block in $V[K_0]$, and that the existence of the later $\varphi$-block in $V'[K_0]$ implies the existence of a similar $\varphi$-block in $V'$.

In $M'$, there is a $\varphi$-block $I \subseteq j\image \mathcal{A}$ of size $\delta$. Note that all its elements are of the form $j(a)$ for some $a\in V'$. Since $M' \subseteq V'[K_0][K_1]$, there is a $\varphi$-block in $V[K_0][K_1]$. In $V[K_0]$, let $\langle \dot{x}_i \colon i < \delta\rangle$ be a sequence of $\mathbb{Q}_1$-names such that $V[K_0][K_1]\models\text{``} \{ j(\dot{x}_i^{K_1}) \colon i < \delta\}$ is a $\varphi$-block''. In $V[K_0]$ let us construct a decreasing sequence of conditions $\langle q_i \colon i < \delta\rangle\subseteq \mathbb{Q}_1$ such that $q_0 \Vdash$``$ \{j(\dot{x}_i) \colon i < \delta\}$ is a $\varphi$-block in $j\image \mathcal{A}$'' and $q_i \Vdash \dot{x}_i = \check{a}_i$, for some $a_i \in V'$. We claim that $\{a_i \colon i < \delta\}$ is a $\varphi$-block for $\mathcal{A}$ in $V'$ (with parameter $p$). For every $\alpha_0 < \alpha_1 < \cdots < \alpha_{n-1} < \delta$, the condition $q = q_{\alpha_{n-1}+1}$ forces $j(a_{\alpha_0}), j(a_{\alpha_1}), \dots, j(a_{\alpha_{n-1}})\in I$. Thus, $q \Vdash j\image \mathcal{A} \models \varphi(j(a_{\alpha_0}), \dots, j(a_{\alpha_{n-1}}), j(p))$. This is a proper subformula of $\Phi$ and thus, by the induction hypothesis, \[q \Vdash M ' \models j(\mathcal{A}) \models \varphi(j(a_{\alpha_0}), \dots, j(a_{\alpha_{n-1}}), j(p)).\] By elementarity, \[q \Vdash V' \models \mathcal{A} \models  \varphi(a_{\alpha_0}, \dots,a_{\alpha_{n-1}}, p).\]
This statement is about the ground model $V'$, so it does not depend on the condition $q$.
We conclude that in $V[K_0]$ there is a $\varphi$-block, $I'\subseteq \mathcal{A}$ of cardinality $\delta$.

Work in $V'$. Let $\langle \dot{y}_i \colon i < \delta\rangle$ be a $\mathbb{Q}_0$-name such that $\langle a_i \colon i < \delta\rangle$ is its $K_0$ realization. In $V'$, let us pick conditions $\langle r_i \colon i < \delta\rangle \subseteq \mathbb{Q}_0$ such that $r_i \Vdash \dot{y}_i = \check{b}_i$ for some $b_i\in V'$, and each $r_i$ forces that $\{ \dot{y}_i \colon i < \delta\}$ is a $\varphi$-block.

By the $\delta$-precaliber of $\mathbb{Q}_0$, there is $X \in [\delta]^\delta$ such that $\{ r_\alpha : \alpha \in X \}$ generates a filter.  The set $I'' = \{ b_\alpha : \alpha \in X \}$ is a $\varphi$-block for $\mathcal A$ in $V'$:
indeed, if $\alpha_0 < \alpha_1 < \cdots < \alpha_{n-1} \in X$, then there is a condition $r\in \mathbb{Q}_0$ stronger than all the conditions $r_{\alpha_1}, \dots, r_{\alpha_{n-1}}$.
$r \Vdash V' \models \mathcal{A} \models \varphi(b_{\alpha_1}, \dots, b_{\alpha_{n-1}}, p)$.
But this is a statement about the ground model, so it does not depend on the condition $r$. We conclude that:
\[V' \models ``\mathcal A \models  Q^n  x_0, \dots, x_{n-1} \varphi(x_0, \dots, x_{n-1}, p),"\]
so by elementarity, $M' \models ``j(\mathcal{A}) \models \Phi$.''
\end{proof}

Applying the reflection argument of Lemma \ref{below huge}, we conclude that the measure $\mathcal{U}$ generated from the huge embedding contains a set $A$ such that every pair of elements $\alpha < \beta$ in $A$ satisfies the conclusion of Lemma \ref{below huge}, when replacing the Chang's relation $\chang$ with the stronger relation $\chang_{Q^{{<}\omega}}$.

Let us look at the model after the iteration of the Easton collapses. We can't conclude that the stronger version of Claim \ref{cc between all successors of regulars} holds, since we do not have a preservation lemma similar to Lemma \ref{preserve closed} for Magidor-Malitz reflection. Thus we can only conclude the following version:

\begin{claim}\label{claim: mm reflection in successor of regulars}
Let $\mathbb{P}_\kappa$ be the iteration defined in \ref{cc between all successors of regulars}.
If $\alpha > \beta$ are regular in the generic extension by $\mathbb{P}_\kappa$ then there is an $\alpha^+$-c.c.\ complete subforcing $\mathbb{P}_i$ such that $\Vdash_{\mathbb{P}_i} \alpha^{+}\chang_{Q^{{<}\omega}} \beta^+$ and $\mathbb{P}_\kappa / \mathbb{P}_i$ is $\alpha^{+}$-closed. The same holds when replacing $\mathbb{P}_\kappa$ by $j(\mathbb{P}_\kappa)$.
\end{claim}

\subsection{Radin forcing}
Work in the model of Claim~\ref{claim:preserving almost hugeness}. In this model, $\GCH$ holds high above $\kappa$, and  $\kappa$ is almost-huge. In particular, $\kappa$ is measurable and $o(\kappa) = \kappa^{++}$ so we can force with the Radin forcing, while collapsing the cardinals between points in the Radin club. We will show that in the generic extension, for every $\mu < \nu < \kappa$, where $\nu$ is a successor, $(\nu^+,\nu)\chang (\mu^+,\mu)$ holds.

We start with a pair of preservation lemmas:
\begin{lemma}\label{lemma: small forcing force cc after q2}
Let $\alpha < \beta$ be regular cardinals such that $\beta^{<\beta} = \beta$, and assume that $\beta^+\chang_{Q^2} \alpha^+$. Assume that $\mathbb{Q}$ is a $\beta$-c.c.\ forcing notion of size $\leq \beta$, and $\mathbb{Q}$ preserves $\alpha^+$. Then $\mathbb{Q}$ forces $(\beta^+, \beta)\chang (\alpha^+, |\alpha|)$.
\end{lemma}
\begin{proof}
Let $\dot{f}$ be a name of a function from $(\beta^+)^{<\omega}$ to $\beta$, such that it is forced that there is no set $A$ of cardinality $\alpha^+$ such that $|\dot{f}\image A^{<\omega}| \leq \alpha$.

Let $\mathcal{A}$ be an elementary substructure of $H(\chi)$ of cardinality $\beta^+$, $\chi$ large enough, $\dot{f}, \mathbb{Q}\in \mathcal{A}$ and let $\mathcal{B} \prec_{Q^2} \mathcal{A}$, $|\mathcal{B}| = \alpha^+$, and $\dot{f}, \mathbb{Q}\in \mathcal{B}$. Let us look at the elements $\dot{f}(a)$, $a\in (\mathcal{B}\cap \beta^+)^{<\omega} = \mathcal{B} \cap (\beta^{+})^{{<}\omega}$. It is forced that there are $\alpha^+$ many elements $a\in \mathcal{B}$ with different realizations for $\dot{f}(a)$.

Thus, in the generic extension, there is a set $\dot{I}\subseteq \mathcal{B}$ such that every pair of elements of it obtain different values under $\dot{f}$. We claim that in the ground model, one can find a set of full cardinality $I \subseteq \mathcal{B}$ such that for every $a, b\in I$ there is a condition $q\in \mathbb{Q}$ that forces $f(a) \neq f(b)$.

Let us look at the collection of all subsets of $I^\prime \subseteq \mathcal{B} \cap (\beta^{+})^{{<}\omega}$ such that for every pair of distinct elements $a, b\in I^\prime$ there is a condition $q\in \mathbb{Q}$, such that $q\Vdash \dot{f}(\check a) \neq \dot{f}(\check b)$. Let $I_m$ be maximal with respect to this condition. If $|I_m| \leq \alpha$, then in the generic extension for every $a\in \mathcal{B}\cap (\beta^{+})^{{<}\omega}$, $\dot{f}(a) \in \dot{f}\image I_m$. If $a\in I_m$, this is clear, and otherwise, there is no condition $q$ that forces it to be different from every element in this set, so every condition forces it to be equal to one of them. In particular, in the generic extension $\alpha^{+} = |\dot{f} \image (\mathcal{B} \cap (\beta^{+})^{{<}\omega})| \leq |I_m| \leq |\alpha|$ - a contradiction to the assumption that $\alpha^{+}$ is not collapsed in the generic extension.

Let $I\subseteq \mathcal{B} \cap (\beta^{+})^{{<}\omega}$ be a set of cardinality $\alpha^{+}$ such that for every $a, b\in I$, there is $q\in \mathbb{Q}$ that forces $\dot{f}(a) \neq \dot{f}(b)$.
By elementarity, for every $a, b\in I$, there is $q\in \mathbb{Q} \cap \mathcal{B}$ forcing the same statement. Therefore $\mathcal{B}$ satisfies the following $Q^2$-sentence:
\[Q^2 a, b\in \mathcal{B}\, \exists q\in \mathbb{Q},\, q\Vdash \dot f(a) \neq \dot f(b)\]
Thus, $\mathcal{A}$ satisfies the same formula: There is a set $I\subseteq \mathcal{A}$, $|I| = \beta^+$, and for every $a, b\in I$, there is $q\in \mathbb{Q}$ such that $q \Vdash \dot{f}(a) < \dot{f}(b)$ or $q \Vdash \dot{f}(a) > \dot{f}(b)$. This defines a coloring $h\colon [\beta^+]^2 \to 2 \times \mathbb{Q}$.

By Remark \ref{remark: erdos-rado for 2k colors}, there are sequences $\langle a_i : i < \beta + 1 \rangle \subseteq (\beta^+)^{<\omega}$ and $\langle q_i : i < \beta \rangle \subseteq \mathbb Q$ such that for all $i < \beta$:
\[ (\forall j >i)\, q_i \Vdash \dot f(a_i) < \dot f(a_j) < \beta, \text{ or }
(\forall j >i) q_i \Vdash \beta > \dot f(a_i) > \dot f(a_j). \]
For each $i$, the second option is impossible by well-foundedness.
By the $\beta$-c.c., there is some $q$ forcing $\beta$-many of $q_i$ to be in the generic filter.  Then $q$ forces that there is an increasing sequence of order type $\beta +1$ below $\beta$, which is impossible.
\end{proof}

\begin{corollary}\label{cor: small forcing force cc after half q2}
Work in the generic extension by $\mathbb{P}_\kappa$. Let $\alpha < \gamma_0 \leq \gamma_1 \leq \beta < \kappa$ be regular cardinals. Assume that
$\mathbb{Q}$ is a $\beta$-c.c.\ forcing notion of size $\leq \beta$, and $\mathbb{Q} \times \Col(\gamma_0,\gamma_1)$ preserves $\alpha^+$.
Then $\mathbb{Q} \times \Col(\gamma_0,\gamma_1)$ forces $(\beta^+, \beta)\chang (\alpha^+, |\alpha|)$.
\end{corollary}
\begin{proof}
Let $\mathbb{P}_i$ be a $\beta^+$-c.c.\ regular subforcing of $\mathbb{P}_\kappa$ such that $\Vdash_{\mathbb{P}_i} \beta^+\chang_{Q^{{<}{\omega}}} \alpha^+$ and $\mathbb{P}_\kappa / \mathbb{P}_i$ is $\beta^+$-closed.
By Claim \ref{claim: mm reflection in successor of regulars}, $\beta^+ \chang_{Q^{<\omega}} \alpha^+$ holds in $V^{\mathbb P_i * \dot \Col(\gamma_0,\gamma_1)}$.
Lemma \ref{lemma: small forcing force cc after q2} implies that $\mathbb Q$ forces $(\beta^+, \beta)\chang (\alpha^+, |\alpha|)$ over this model.
In $V^{\mathbb P_i}$, $\mathbb{P}_\kappa / \mathbb{P}_i$ is $\beta^+$-closed, and $| \mathbb Q \times \Col(\gamma_0,\gamma_1) | \leq \beta$.  Therefore, Lemma \ref{preserve closed} implies that in the generic extension Chang's Conjecture, $(\beta^+, \beta)\chang (\alpha^+, |\alpha|)$, holds.
\end{proof}
We are now ready for the main theorem. We start by defining a notion of Radin forcing with interleaved collapses. For simplicity, we assume $\GCH$.

Recall the following definition of a measure sequence, due to Radin.
\begin{definition}\cite{Radin1982}
Let $j\colon V\to M$ be an elementary embedding, $\crit j = \alpha$. Let us define, by induction, a sequence of normal measures on $V_\alpha$, $u$. Let $u(0) = \alpha$.

For every $i < j(\alpha)$, if $u\restriction i \in M$, let $u(i) = \{X \subseteq V_\alpha \colon u\restriction i \in j(X)\}$. Otherwise, we halt.

$u$ is called the \emph{measure sequence} derived from $j$.
\end{definition}

Let $j\colon V\to M$ be an elementary embedding with critical point $\kappa$. Let $\mathcal{U}$ be the measure sequence derived from $j$. Let $MS$ be the class of all measure sequences.

We say that a measure $u$ on $V_\alpha$ is \emph{normal} if it is closed under diagonal intersections in the following sense: if $\langle A_v \mid v \in MS \cap V_\alpha\rangle$ is a list of sets from $u$ then also:
\[\triangle_{v} A_v = \{x \in V_\alpha \cap MS \mid \forall v\in V_{x(0)} \cap MS,\, x \in A_v\}\]

Let us start with the following fact:
\begin{lemma}\label{lemma: guiding generics}
There is a sequence $\langle \guidinggeneric_\alpha \colon \alpha \leq \kappa\rangle$ such that:
\begin{enumerate}
\item For every measurable $\alpha$, and every $f\in \guidinggeneric_\alpha$, $f$ is a function with domain $V_{\alpha}$ and for every measure sequence $u \in \dom f$, $f(u) \in \Col(u(0)^+, \alpha)$.
\item Let $u$ be a normal measure on $V_\alpha$, $\alpha \leq \kappa$, and let $j_u\colon V\to M_u$ be the ultrapower embedding. Then the set $\{[f]_u \colon f\in \guidinggeneric_\alpha\}$ is an $M_u$-generic filter for $\Col(\alpha^{+}, j_u(\alpha))$. Moreover, for every function $D\colon V_\alpha \to V$, such that for every measure sequence $v\in V_\alpha$, $D(v)$ is a dense open subset of $\Col(v(0)^+, \alpha)$, there is $f\in \guidinggeneric_\alpha$ such that $\{v \in MS \cap V_\alpha \colon f(v)\in D(v)\}$ belongs to every normal ultrafilter on $V_\alpha$.
\end{enumerate}
\end{lemma}

\begin{proof}
Let $\delta < \kappa$.  Using \GCH, enumerate all the functions $D : V_\delta \to V_{\delta+1}$, such that  $D(u)$ is a dense open subset of $\Col(u(0)^+, \delta)$ for every measure sequence $u\in V_\delta$, as $\langle D_\alpha : \alpha < \delta^+ \rangle$.  Let $p_0 : V_\delta \to V_\delta$ be such that $p_0(u) \in D_0(u)$ for every measure sequence $u \in V_\delta$.  Given $\langle p_i : i \leq \alpha \rangle$, $\alpha < \delta^+$, let $p_{\alpha+1}$ be such that $p_{\alpha+1}(u) \leq p_\alpha(u)$ and $p_{\alpha+1}(u) \in D_{\alpha+1}(u)$ for all measure sequences $u$.  At limit stages $\lambda$ in the construction, we use the following inductive assumption:  For every $\alpha < \beta < \lambda$, there is a club $C_{\alpha,\beta} \subseteq \delta$ such that whenever $u(0) \in C_{\alpha,\beta}$ for a measure sequence $u$, $p_\alpha(u) \geq p_\beta(u)$.  Let $\langle \lambda_\alpha : \alpha < \cf(\lambda) \rangle$ be increasing and cofinal in $\lambda$.  The diagonal intersection, $C = \{ \alpha < \delta :$ for all $\beta < \gamma < \alpha$, $\alpha \in C_{\lambda_\beta,\lambda_\gamma} \}$, is club.  For all $u$ such that $u(0) \in C$, $\langle p_{\lambda_\alpha}(u) : \alpha < u(0) \rangle$ is a decreasing sequence in $\Col(u(0)^+,\delta)$.   Let $p_\lambda$ be such that $p_\lambda(u) \in D_\lambda(u)$ is a lower bound to this sequence for all such $u$.  To continue the induction, we define $C_{\lambda_\alpha,\lambda} = C \cap \{ \beta \mid \alpha < \beta \}$ for $\alpha < \cf(\lambda)$.  For $\beta < \lambda$ not among the $\lambda_\alpha$, let $C_{\beta,\lambda} = C_{\lambda_\alpha,\lambda} \cap C_{\beta,\lambda_\alpha}$, where $\alpha$ is the least ordinal such that  $\lambda_\alpha > \beta$.

For every normal measure $u$ on $V_\delta$, $\{[p_\alpha]_u \colon \alpha < \delta^+ \}$ is a descending sequence in $\Col(\delta^+,j_u(\delta))$, and $[p_\alpha]_u \in [D_\alpha]_u$ for every $\alpha < \delta^+$.  We let $\guidinggeneric_\delta = \{ p_\alpha : \alpha < \delta^+ \}$.  Finally, we let $\guidinggeneric_\kappa = j(\langle \guidinggeneric_\alpha : \alpha < \kappa \rangle)(\kappa)$.
\end{proof}

\begin{lemma}\label{lemma: guiding generics reflection}
Under the same assumptions:
\begin{enumerate}
\item Assume that $\langle f_\alpha \colon \alpha < \kappa\rangle$ is a sequence of functions, $f_\alpha \in \guidinggeneric_\alpha$. Then there is a function $f \in \guidinggeneric_\kappa$ such that the collection: $\{v \in V_\kappa \colon f \restriction V_{v(0)} = f_{v(0)}\}$ is in $\bigcap_{0 < \beta < \len \mathcal{U}} \mathcal{U}(\beta)$.
\item Let $B$ be the set of all measure sequences $u \in V_\kappa$ such that for every $\langle f_\gamma \colon \gamma < u(0)\rangle$, with $f_\gamma \in \guidinggeneric_\gamma$ there is $f$ such that $\{v \in V_{u(0)} \colon f \restriction V_{v(0)} = f_{v(0)}\}$ is in $\bigcap_{0 < \beta < \len u} u(\beta)$. Then $B \in \bigcap_{0 < \beta < \len \mathcal{U}} \mathcal{U}(\beta)$.
\end{enumerate}
\end{lemma}

\begin{proof}
For (1), let $f = j(\langle f_\alpha \colon \alpha < \kappa\rangle) \in \guidinggeneric_\kappa$.  Note that $j(f) \restriction V_\kappa = f$.  Let $A = \{ \alpha : f \restriction V_\alpha = f_\alpha \}$.  The set $A'$ of measure sequences $u \in V_\kappa$ such that $u(0) \in A$ is in $\mathcal U(\beta)$ for every $\beta < \len \mathcal U$, since $\kappa \in j(A)$.

Now let $B$ be as in (2).  For all $\beta < \len \mathcal U$, $M$ can see that for all sequences $\langle f_\alpha \colon \alpha < \kappa\rangle$ as in (1), there is $f \in \guidinggeneric_\kappa$ such that $\{ v \in V_\kappa : f \restriction V_{v(0)} = f_{v(0)} \} \in \bigcap_{\alpha < \beta} \mathcal U(\alpha)$.  Thus $\mathcal U \restriction \beta \in j(B)$, and $B \in \mathcal U(\beta)$.
\end{proof}

Let us define the forcing notion $\mathbb{P}$. $p\in \mathbb{P}$ iff
\[p = \langle f_{-1}, q_0, f_0, \dots, q_n\rangle\]
Where
\[q_i = \langle u_i, A_i, F_i\rangle\]
and:
\begin{enumerate}
\item $u_i$ is a measure sequence. We denote $u_i(0)$ by $\alpha_i$.
\item If $\len u_i > 0$, $A_i \in \bigcap_{0 < \beta < \len(u_i)} u_i(\beta)$. Otherwise $A_i = \emptyset$.
\item $F_i \colon A_i \to V$, and for all $\beta \in A_i$, $F_i(\beta) \in \Col(\beta^+, \alpha_i)$ and $F_i \in \guidinggeneric_{u_i(0)}$.
\item $u_n = \mathcal{U}$.
\item For every $i \geq 0$, $f_i \in \Col(\alpha_i^{+}, \alpha_{i+1})$.
\item $f_{-1} \in \Col(\omega, \alpha_0)$.
\item If $v\in A_i$ then $v(0) > \sup \range f_{i-1}$.
\end{enumerate}

Let
\[p = \langle f_{-1}, u_0, A_0, F_0, f_0, \dots, u_n, A_n, F_n\rangle \in \mathbb{P}\]
\[p^\prime = \langle f^\prime_{-1}, u^\prime_0, A^\prime_0, F^\prime_0, f^\prime_0, \dots, u^\prime_m, A^\prime_m, F^\prime_m\rangle \in \mathbb{P}\]

$p \leq p^\prime$ iff:
\begin{enumerate}
\item $m \leq n$ and there is a strictly increasing sequence of indices $i_0, \dots, i_m$ such that $u_{i_j} = u^\prime_j$. Let us set $i_{-1} = -1$.
\item For all $-1 \leq j \leq m$, $f_{i_j} \supseteq f^\prime_j$.
\item $A_{i_j} \subseteq A^\prime_j$.
\item For every $-1 \leq j < m$ and every $i_j < k < i_{j+1}$, $u_k \in A_{i_{j+1}}$, $f_k \supseteq F^\prime_{i_{j+1}}(u_k)$ and $A_k \subseteq A^\prime_{j+1}$.
\item For all $j$ and $i_j < k \leq i_{j+1}$, for all $\beta \in A_k$, $F_k(\beta) \supseteq F^\prime_{i_{j+1}} (\beta)$.
\end{enumerate}

We say that $p \leq^\star q$ if $p\leq q$ and $\len p = \len q$.

For a condition $p = \langle f_{-1}, u_0, A_0, F_0, f_0, \dots, u_n, A_n, F_n\rangle \in \mathbb{P}$, let us denote:
\[\stem p = \langle f_{-1}, u_0, A_0, F_0, f_0, \dots, u_n\rangle\]

\begin{lemma}\label{lemma: radin centered}
$\mathbb{P}$ is $\kappa$-centered.
\end{lemma}

For any measure sequence $u$ which is derived from some elementary embedding, let us denote by $\mathbb{P}^{u}$ the forcing notion which is defined as $\mathbb{P}$ when replacing $\mathcal{U}$ by $u$. Lemma \ref{lemma: radin centered} holds for $\mathbb{P}^u$. We have the standard decomposition:
\begin{claim}
For every condition $p\in \mathbb{P}$ of length $n$, and every measure sequence $u$ in the stem of $p$, the forcing $\mathbb{P} \restriction p$ of all conditions below $p$ splits into the product:

$\mathbb{P} \restriction p = \mathbb{P}^{>u} \restriction p_{\uparrow} \times \mathbb{P}^{u} \restriction p_{\downarrow}$, where $\mathbb{P}^{>u}$ is the forcing $\mathbb{P}$ when we modify the definition of $\alpha_{-1}$ to be $u(0)^+$. $p_{\uparrow}, p_{\downarrow}$ is the decomposition of the condition $p$ to the parts above and below $u$, respectively.
\end{claim}

\begin{lemma}
$\mathbb{P}$ satisfies the Prikry property. Moreover, this is true for $\mathbb{P}^{>u}$ for every measure sequence $u$.
\end{lemma}
\begin{proof}
We sketch a proof for the lemma. The proof is similar to the one in \cite[Section 3, Section 4]{Gitik2010}, with minor changes. We will prove it only for the case of $\mathbb{P}$. The other cases are similar.
Let $\Phi$ be a statement in the forcing language. Let $p\in \mathbb{P}$ be a condition.

In order to show that $\mathbb{P}$ satisfies the Prikry property, we will show that this is true for $\mathbb{P}^u$, for every measure sequence $u$.

Let us assume, by induction, that this is true for every measure sequence $u$ such that $u(0) < \kappa$.

Let us start with the case $\len p = 0$:

\begin{claim}
Assume that $p = \langle f_{-1}, \mathcal U, A, F\rangle$. There is a direct extension of $p$ that decides the truth value of $\Phi$.
\end{claim}
\begin{proof}
First, let us consider, for every stem $s$, the following sets:
\[D^0_{s, v} = \{g \leq F(v) \colon \exists A_v, F_v, A^\prime, F^\prime, s^\smallfrown \langle A_v, F_v, g, A^\prime, F^\prime\rangle \parallel \Phi\}\]
\[D^1_{s, v} = \{g \leq F(v) \colon \forall A_v, F_v, A^\prime, F^\prime, \forall g' \leq g,\, s^\smallfrown \langle A_v, F_v, g', A^\prime, F^\prime\rangle \not\parallel \Phi\}\]
Clearly, $D^0_{s, v} \cup D^1_{s,v}$ is a dense subset of $\Col(v(0)^{+},\kappa)$.
By the distributivity of $\Col(v(0)^{+},\kappa)$, the intersection $D_v = \bigcap_{s\in V_{v(0)}} (D^0_{s, v} \cup D^1_{s,v})$ is a dense subset of $\Col(v(0)^{+},\kappa)$. By Lemma \ref{lemma: guiding generics}, there is $F'\in \guidinggeneric_\kappa$ such that the set of all $v\in V_\kappa$ with $F'(v)\in D_v$ belongs to $\bigcap_{0 < \alpha < \len \mathcal{U}} \mathcal{U}(\alpha)$. Let $A'$ be the intersection of the above set with $A$. Let $F^\star$, be a condition in $\guidinggeneric_\kappa$ stronger than $F, F'$.

Let us define for every possible stem of a condition stronger than $\langle f_{-1}, \mathcal{U}, A', F^\star\rangle$,
\[s = \langle f^s_{-1}, u_0^s, A_0^s, F_0^s, f_0^s, \dots, u_{k-1}^s, A_{k-1}^s, F_{k-1}^s, f_{k-1}^s \rangle,\]
and for every $\alpha < \len \mathcal{U}$, a set $A(s, \alpha)\in \mathcal{U}(\alpha)$. This is a measure one set, relative to $\mathcal{U}(\alpha)$, such that one of the three possibilities holds for it:
\begin{enumerate}
\item For every $v\in A(s,\alpha)$ there is a choice of $B^s_v, F^s_v, f^s_v$ such that an extension of $p$ with the stem $s^\smallfrown \langle v, B^s_v, F^s_v, f^s_v\rangle$ forces $\Phi$.
\item For every $v\in A(s,\alpha)$ there is a choice of $B^s_v, F^s_v,f^s_v$ such that an extension of $p$ with the stem $s^\smallfrown \langle v, B^s_v, F^s_v,f^s_v\rangle$ forces $\neg\Phi$.
\item For every $v\in A(s,\alpha)$, there is no extension of $p$ with stem $s ^\smallfrown \langle v, B_v, F_v,f_v\rangle$, that forces either $\Phi$ or $\neg\Phi$.
\end{enumerate}
Using the closure of the generic filter $\guidinggeneric_{v(0)}$, we may assume that $F^s_v, f^s_v$ depend only on $v$ (by taking the lower bound of all the $F^s_v, f^s_v$ with $s \in V_{v(0)}$ - there are only $v(0)$ many such stems).
Let $A(\alpha)$ be the diagonal intersection of all the $A(s, \alpha)$, and let $A^\star = A' \cap \bigcup_{\alpha < \len \mathcal{U}} A(\alpha)$. Let $p^\star = \langle f_{-1}, \mathcal U, A^\star, F^\star\rangle$.

Let us observe first that for every $v\in A(s, \alpha) \cap A'$, if one of the first two options holds, then we may take $f_v = F^\star(v)$. Recall that $F^\star(v) \in D^0_{s, v} \cup D^1_{s,v}$.  $f_v \leq F^\star(v)$, and it decides the truth value of $\Phi$. Thus, $F^\star(v) \in D^0_{s, v}$. In particular, there are $B_v^\prime, F_v^\prime$ that together with $F^\star(v)$ decide the truth value of $\Phi$. By compatibility, a condition with stem $s^\smallfrown \langle v, B_v \cap B_v^\prime, F_v \wedge F^\prime_v, F^\star(v)\rangle$ must force the same truth value for $\Phi$ as a condition with stem $s^\smallfrown \langle v, B_v, F_v, f_v\rangle$.

Let us take an extension of $p^\star$ which decides $\Phi$ and has a minimal length. If it is a direct extension, we are done. Let us assume, towards a contradiction, that this extension has length $n + 1$:
\[r = \langle f_{-1}^r, v^r_0, A^r_0, F^r_0,f^r_0 \dots, \mathcal U, A^r_{n+1}, F^r_{n+1}\rangle\]
Let $s$ be the lower stem (up to length $n$). By our assumption, there is $\alpha$ such that $A(s, \alpha)\in \mathcal{U}(\alpha)$ contains only measure sequences $v$, which when appended to $s$ together with $A_v, F_v$, form a condition that decides the statement in the same direction as $r$. Without loss of generality, they all force $\Phi$.

For every $v \in A(s, \alpha)$, $F_v$ is stronger than the restriction of $F^\star$ to $A_v$ (pointwise) and belongs to $\guidinggeneric_{v(0)}$.
Let us consider the function $g\colon A(s, \alpha) \to V_\kappa$, $g(v) = \langle A_v, F_v\rangle$. By the definition of $\mathcal U(\alpha)$, $\mathcal{U}\restriction \alpha \in \dom j(g)$. Let $\langle A^{<\alpha}, F^{<\alpha}\rangle = j(g)(\mathcal{U}\restriction \alpha)$.

By elementarity, $A^{<\alpha} \in \bigcap_{\beta < \alpha} \mathcal{U}(\beta)$ and for $\mathcal{U}(\alpha)$-almost all $v \in V_\kappa$, $A^{<\alpha} \cap V_{v(0)} = A_v$ and $F^{<\alpha} \restriction V_{v(0)} = F_v$. Let $A^{\alpha}$ be the collection of all $v \in A(s,\alpha)$ that satisfy the above assertion. By Lemma \ref{lemma: guiding generics reflection}, $F^{<\alpha}\in \guidinggeneric_\kappa$, so let $F^{\star\star} = F^{<\alpha} \wedge F^\star$.

Let $A^{>\alpha}$ to be all the sets that reflect $A^{\alpha}$, namely $A^{>\alpha} = \{u\in A^\star \colon \exists \beta,\, A^{\alpha} \cap V_{u(0)} \in u(\beta)\}$. Now let $A^{\star\star} = A^{<\alpha} \cup A^\alpha \cup A^{>\alpha}$, and let us restrict the domain of $F^{\star\star}$ to $A^{\star\star}$.

Let us show that any extension of the condition $q_s = s ^\smallfrown \langle\mathcal{U}, A^{\star\star}, F^{\star\star}\rangle$ is compatible with a choice of an element from $A(s, \alpha)$. Therefore, any extension of the current condition is compatible with an extension that forces $\Phi$.

This is true by our choice of $A^{\star\star}$. If we extend $q_s$ by only adding elements below $v_n$ and strengthening the collapses, then this condition is compatible with any condition in which we extend $q_s$ by adding a single element from $A(s, \alpha)$ above $v_{n-1}$. Otherwise, let $q \leq q_s$ be any extension of $q_s$ and assume that the Radin club of $q$ contains elements above $v_{n-1}$. Let $m = \len q$,
\[q = \langle f_{-1}^q, u_0^q, A_0^q,F_0^q,f_0^q, \dots, u_m^q, A^q_m, F^q_m \rangle\]
and assume that $k < m$ is the first index of an element in the Radin sequence which
is a measure sequence of length $>0$ such that $(A^{\alpha} \cup A^{> \alpha}) \cap V_{u_k(0)} \in \bigcup_{\beta < \len u_k} u_k(\beta)$. If there is no such element, let us pick any nontrivial measure sequence $v \in A^\alpha \cap A^q_m$.  Then $A^{<\alpha} \cap A^q_i \in \bigcap_{\beta < \len u_i} u_i(\beta)$ for all $i < m$, and $A_v = A^{<\alpha} \cap V_{v(0)}$.  Thus adding $\langle v,A_v,F^{\star\star} \restriction A_v \rangle$ to $q_s$ results in a condition that forces $\Phi$ and is compatible with $q$.

So, let us assume that there is such element $u_k^q$. If $A^{\alpha} \in u_k^q(\beta)$ for some $\beta < \len u_k^q$, then there is $v \in A^q_k$ such that $A_v \cap A^q_k \in \bigcap_{\beta < \len v} v(\beta)$, so as above, $q_s$ may be extended by $\langle v,A_v,F^{\star\star} \restriction A_v \rangle$ to get a condition compatible with $q$ that forces $\Phi$.
If $A^{>\alpha} \in u_k^q(\beta)$ for some $\beta < \len u_k^q$, then by our choice of $A^{>\alpha}$ there is some $v \in A^q_k$ that can be added to $q$ to put us into the previous case.

We conclude that in any case, any extension of $q_s$ has an extension that forces $\Phi$ and thus, $q_s \Vdash \Phi$. But this is a contradiction to the minimality of $n$.
\end{proof}
Let us continue to the general case.

Let $p\in \mathbb{P}$ be a general condition. Let us assume, by induction, that Prikry property holds for every shorter condition. By the claim above, we may assume that $\len p > 0$.
We want to find a direct extension of $p$ that decides the truth value of statement $\Phi$.

We can decompose the forcing notion $\mathbb{P} \restriction p$ into a product $\mathbb{P}^{>u}\restriction p_\uparrow \times \mathbb{P}^u \restriction p_\downarrow$ for some measure sequence $u$ that appears in $p$.

Recall that $\mathbb{P}^u$ is $\alpha$-centered, where $\alpha = u(0)$. Let $\langle r_i \colon i < \alpha\rangle$ enumerate all possible stems of conditions in $\mathbb{P}^u$.

Let us define, by induction, a sequence of conditions in $\mathbb{P}^{>u}$, $\langle p_i \colon i \leq \alpha\rangle$ in the following way. Let $p_0 = p_\uparrow$.  Given $p_i$, let $p_{i+1} \leq^\star p_i$ decide whether there is a condition in $\mathbb P^u$ with stem $r_i$ deciding $\Phi$, and if so, whether it forces $\Phi$ or $\neg \Phi$.  At limit ordinals $i \leq \alpha$, we use the closure of the order $\leq^\star$ and take $p_i$ to be a lower bound of $p_{i'}$ for every $i' < i$.

If $G \times H \subseteq \mathbb P^u \times \mathbb{P}^{>u}$ is generic with $\langle p_\downarrow,p_\alpha \rangle$, then there is a condition $\langle r,q \rangle \in G \times H$ deciding $\Phi$.  The stem of $r$ is $r_i$ for some $i < \alpha$, and $p_\alpha$ must already decide which way $r$ decides $\Phi$.  Thus it is forced by $\mathbb P^u$ that $p_\alpha$ decides $\Phi$.  By induction, we may take a direct extension $r' \leq^\star p_\downarrow$ which decides which way $p_\alpha$ decides $\Phi$.  $r^{\prime \smallfrown} p_\alpha$ is the desired direct extension of $p$.
\end{proof}

Recall that $\mathcal{U}$ was derived from an elementary embedding $j\colon V\to M$.
\begin{lemma}
Assume that $V_{\kappa + 2} \subseteq M$. Then $\mathbb{P}$ preserves the measurability of $\kappa$.
\end{lemma}
\begin{proof}
By $\GCH$ and the strength of the elementary embedding, $j$, we have $\len \, \mathcal{U} \geq \kappa^{++}$, while $|V_{\kappa + 1}| = \kappa^{+}$. Therefore, there is a \emph{repeat point}, namely an ordinal $\alpha < \kappa^{++}$ such that \[\bigcap_{0 < \beta < \alpha} \mathcal{U}(\beta) = \bigcap_{0 < \beta < \kappa^{++}} \mathcal{U}(\beta).\] Let $\alpha$ be the first repeat point in the measure sequence derived from $j$.

Clearly, replacing $\mathcal{U}$ with $\mathcal{U}\restriction \alpha$ does not change the forcing. Let $k\colon V \to M$ be an elementary embedding generated by $\mathcal{U}(\alpha)$. Let us look at $k(\mathbb{P})$. Let:
\[p = \langle f_{-1}, u_0, A_0, F_0, f_0, \dots, u_n, A_n, F_n\rangle \in \mathbb{P}.\]
Let us extend the condition $k(p)$ by adding $\langle \mathcal{U}\restriction \alpha, A_n, F_n\rangle$ at the $n$-th coordinate, and let $q$ be the obtained condition. The Radin forcing below the condition $q$ is equivalent to a product $\mathbb{P} \times \mathbb{Q}$ (where $\mathbb{Q}$ consists of all the upper parts of the conditions in $\mathbb{P}$). Recall that $\mathbb{Q}$ is $\kappa^{+}$-weakly closed. 

Let us define an ultrafilter in $V[G]$ by $p \Vdash \dot{A} \in \mathcal{V}$ if and only if there is a direct extension $q^\star$ of the above $q$, where the lower part of $q^\star$ is $p$, and it forces $\check\kappa \in k(\dot{A})$. Since every pair of direct extensions are compatible, $\mathcal{V}$ is well defined. Moreover, using the Prikry Property and the genericity of $G$, for every $\dot{A}$ there is $p\in G$ such that $p \Vdash \dot{A} \in \mathcal{V}$ or $p \Vdash \check\kappa \setminus \dot{A} \in \mathcal{V}$. Let us show that $\mathcal{V}$ is $\kappa$-complete. 

Assume otherwise and let us work, for a moment, in $V$. Let $p$ be a condition that forces that $\langle \dot{A}_i \mid i < i_\star\rangle$ is a sequence of names of elements of $\mathcal{V}$ with intersection not in $\mathcal{V}$ and $i_\star < \kappa$. Let $q$ be as before. 
We conclude that there are direct extensions $q_i \leq^* q$ for each $i < i_\star$ such that $q_i \Vdash \check\kappa \in k(\dot{A}_i)$. 
Since the parts below $\kappa$ of those conditions are the same and they are all a direct extension of the same condition, there is a single condition $q^\star$ stronger than all of them. 
$q^\star \Vdash \forall i < i_\star \check\kappa \in k(A_i)$ and in particular, 
\[q^\star \Vdash \check\kappa \in k(\bigcap _{i < i_\star}\dot{A}_i) = \bigcap_{i < i_\star} k(A_i).\]
\end{proof}
Since $\kappa$ is measurable in the generic extension, it is in particular inaccessible there and thus $V_\kappa$ of the generic extension is a model of $\ZFC$.

There is a natural projection from a measure on measure sequences in $V_\kappa$ to a measure on $\kappa$ by taking each measure sequence $u$ to its first element $u(0)$. When saying that a subset of $\kappa$ is large relative to a measure on the measure sequences of $V_\kappa$ we mean that it is large relative to the corresponding projection.

Let us return now to the model that was obtained in the previous section.
\begin{theorem}
Let $\mathbb{P}$ be the Radin forcing for adding a club through $\kappa$, with interleaved collapses, collapsing $\rho_{i+1}$ to be of cardinality $\rho_i^+$ for any two successive Radin points. Let $A$ be the set obtained in Lemma \ref{below huge}. Assume that $A$ is $\mathcal{U}$-large relative to all relevant measures.

Then $\mathbb{P}$ forces $(\beta^{++}, \beta^+)\chang (\alpha^+, \alpha)$ for all $\alpha \leq \beta < \kappa$.
\end{theorem}
\begin{proof}
Let $p \in \mathbb P$ force that $\beta$ be a cardinal in the extension.  Assume $p$ is strong enough to decide three successive points $\zeta < \xi < \rho$ in the Radin club such that $p \Vdash \beta^+ = (\zeta^+)^V,\, \beta^{++} = (\xi^+)^V$, and $\beta^{+3} = (\rho^+)^V$.

The forcing $\mathbb{P}\restriction p$ splits into a product $\mathbb{Q}_0 \times \Col(\xi^+,\rho) \times \mathbb Q_1$, where $\mathbb{Q}_0$ is the Radin forcing below $\xi$, and $\mathbb Q_1$ adds no subsets of $\xi^+$.

Note that $\mathbb{Q}_0 = \mathbb{P}_0 \times \Col(\zeta^+, \xi)$, where $\mathbb P_0$ is $\zeta^+$-c.c., has size $\leq \zeta^+$, and preserves $\alpha^+$, which is the successor cardinal in $V$ to some member of the Radin club.  Corollary \ref{cor: small forcing force cc after half q2} implies that $(\beta^{++},\beta^+) \chang (\alpha^+,\alpha)$ holds after forcing with $\mathbb{P}_0 \times \Col(\zeta^+, \xi)$.

Since $\mathbb Q_0$ is $\xi^+$-c.c.\ and $\Col(\xi^+,\rho)$ is $\xi^+$-closed, Lemma \ref{preserve closed} implies that the instance of Chang's Conjecture holds after forcing with $\mathbb{Q}_0 \times \Col(\xi^+, \rho)$.  It continues to hold after forcing with $\mathbb Q_1$, since no new algebras on $\xi^+$ are added.
\end{proof}

In the above model, the instances of Chang's Conjecture of the form $(\mu^+, \mu)\chang (\nu^+, \nu)$ where $\mu$ is singular, always fail. Since every singular cardinal in the generic extension is inaccessible in the ground model, ${\Square}_\mu^\star$ holds there. Since we preserve its successor, it still holds in the generic extension. Any instance of Chang's Conjecture of the form $(\mu^+, \mu)\chang (\nu^+, \nu)$, where $\mu$ is singular, implies the failure of the weak square, $\Square_\mu^\star$, by \cite{ForemanMagidor1997}.

\section{Segments of Chang's Conjecture}\label{sec: segment cc}
In the previous section we dealt with obtaining Chang's Conjecture between all pairs of the form $(\mu^+, \mu)$ and $(\nu^+, \nu)$ where $\mu$ is a successor cardinal. The cases of $\mu$ singular cardinal, which were not covered in the previous section, are much harder.

Recall that any instance of Chang's Conjecture of the form $(\mu^+,\mu)\chang(\nu^+,\nu)$ where $\mu$ is singular, $\nu < \mu$ (in which we assume that the elementary submodel of cardinality $\nu^+$ contains $\nu$) implies the failure of the weak square $\square_{\mu}^\star$. Indeed, it implies that there are no good scales. Thus, the problem of getting, for example $(\mu^+,\mu)\chang((\cf \mu)^+, \cf \mu)$ globally requires us to get a failure of weak square at all singular cardinals. See \cite{DiagonalSCPrikry} for the best known consistency result towards this goal.

We want to attack a more modest problem. We will get all the instances of Chang's Conjecture which are compatible with \GCH, in a small segment of cardinals not covered by the previous section.

Let $\kappa$ be a huge cardinal, and let $j\colon V\to M$ be an elementary embedding witnessing it. Let $\delta = j(\kappa)$. 
\begin{lemma}\label{lem: universal laver function}
There is a function $\ell\colon \delta \to V_\delta$ with the following properties:
\begin{enumerate}
\item $j(\ell\restriction \kappa) = \ell$.
\item For every cardinal $\mu$, and $\gamma < \delta$, if $x\in V_\gamma$ and $\mu$ is $2^\gamma$-supercompact then there is an elementary embedding $k\colon V\to N$, $\crit k = \mu$, $k\image \gamma \in N$, $k(\ell)(\mu) = x$.
\item For every $\mu < \kappa$ non-measurable, $\ell(\mu) = \emptyset$.
\item For every $\mu < \kappa$, if $\mu$ is ${<}\delta$-supercomapct, $\ell(\mu) = \emptyset$.
\end{enumerate}
\end{lemma}
\begin{proof}
We pick a universal Laver function $\ell\restriction \kappa$ on $V_\kappa$, using minimal counter examples (if there are counter examples). I.e.\ we pick $\ell(\rho) = x$ if $x\in H(\theta)$ and $\theta < \delta$ is the minimal cardinal for which there is no supercompact measure on $P_\rho \theta$, $\mathcal{U}$ such that $j_{\mathcal{U}}(\ell\restriction \rho) = x$. We pick $x$ to be the minimal such witness of failure, relative to some fixed well order. 
 
Apply $j$ on this function to obtain $\ell$. Since $V_\delta \subseteq M$, the first item holds. The second item holds below $\kappa$ by the usual proof for the Laver diamond: the minimal $x$ for which property (2) fails would be exactly the value of $j(\ell)$ at the point $\mu$ in any closed enough ultrapower. The third item holds by definition. The fourth item holds as for any fully supercompact cardinal (from the point of view of $V_\delta$), $\mu$, $\ell\restriction \mu$ is already a Laver diamond and therefore there is no counter example to pick. 
\end{proof}

Let $\ell$ be a Laver function for $V_\delta$ as in the conclusion of Lemma \ref{lem: universal laver function}.

We wish to make every $\alpha< \kappa$ which is ${<}\kappa$-supercompact indestructible under any $\alpha$-directed-closed forcing.  We do the usual Laver iteration $\mathbb P$ with respect to $\ell$.  We claim that if $G \subseteq j(\mathbb P)$ is generic, then $\kappa$ is still huge in $V[G]$.  Since $\kappa$ is ${<}\delta$-supercompact in $M$, $j(\mathbb P)/(G \cap \mathbb P)$ is $(2^\kappa)^+$-directed-closed in $V[G\cap \mathbb P]$.  Thus we may take a master condition and build an $M[G]$-generic filter for $j^2(\mathbb P)/G$.

Now let $A$ be the set obtained from Lemma \ref{below huge}. By reflection arguments, we may assume every cardinal in $A$ is ${<}\kappa$-supercompact. Let us pick such cardinals $\langle \mu_n \mid n < \omega\rangle$.

\begin{theorem}
There is $\rho < \mu_0$ such that:
\[\Col(\omega, \rho^{+\omega}) \ast \E(\rho^{+\omega+1}, \mu_0) \ast \E(\mu_0, \mu_1) \ast \cdots \ast \E(\mu_n, \mu_{n+1})\ast \cdots\]
forces:
\begin{enumerate}
\item For every $m < n < \omega$, $(\aleph_{n+1}, \aleph_n) \chang (\aleph_{m+1}, \aleph_m)$.
\item $(\aleph_{\omega+1}, \aleph_{\omega}) \chang (\aleph_1, \aleph_0)$.
\end{enumerate}
\end{theorem}
In the generic extension $\aleph_1 = \rho^{+\omega + 1}$ and $\aleph_{n + 2} = \mu_n$ for all $n  <\omega$.

We split the proof of this theorem into three parts:
\begin{lemma}
For every choice of $\rho$, $(\aleph_{2}, \aleph_1)\chang (\aleph_1, \aleph_0)$.
\end{lemma}
\begin{proof}

After forcing with $\Col(\omega, \rho^{+\omega})$, $\mu_0$ is still measurable. Since the forcing $\E(\omega_1, \mu_0)$ is $\mu_0$-c.c.\ and $\sigma$-closed, after forcing with it there is an $\omega_2$-complete ideal on $\omega_2$, in which the positive sets have a dense subset which is $\sigma$-closed. In particular, the Strong Chang conjecture holds (see, for example, \cite{Sakai2005}[Theorem 1.1]).
\end{proof}

\begin{lemma}
For every choice of $\rho$, for every $0 < m < n < \omega$, $(\aleph_{n + 1}, \aleph_n)\chang (\aleph_{m + 1}, \aleph_m)$.
\end{lemma}
\begin{proof}
This is a special case of Lemma \ref{below huge}.
\end{proof}

\begin{lemma}
There is a choice of $\rho$ for which $(\aleph_{\omega+1}, \aleph_{\omega}) \chang (\aleph_1, \aleph_0)$.
\end{lemma}
\begin{proof}
Let us show first that there is $\rho$ such that after forcing with:
\[(\Col(\omega, \rho^{+\omega}) \ast \E(\rho^{+\omega+1}, \mu_0)) \times (\E(\mu_0, \mu_1) \times \cdots \times \E(\mu_n, \mu_{n+1})\times \cdots) \]
Chang's conjecture $(\aleph_{\omega+1}, \aleph_{\omega}) \chang (\aleph_1, \aleph_0)$, holds.

Replace the order of the product and force with the second component first. By the indestructibility of $\mu_0$, $\mu_0$ remains ${<}\kappa$-supercompact after this forcing.
Assume that for every $\rho < \mu_0$ there is a name for a function $\dot{f}_\rho \colon (\mu_0^{+\omega+1})^{<\omega} \to \mu_0^{+\omega}$ witnessing the failure of Chang's conjecture in the generic extension.

Let $k\colon V \to M$ be an elementary embedding with critical point $\mu_0$ and $A = k\image \mu_0^{+\omega+1} \in M$.  Let $\dot f = k(\langle \dot f_\rho : \rho < \mu_0 \rangle)(\mu_0)$.  By our assumption, $\Vdash |\dot{f} \image A^{<\omega}| = \aleph_1$. Since the image of $\dot{f}$ is contained in $k(\mu_0^{+\omega})$, there is an integer $n >0$ and a condition $p \in  \Col(\omega, \mu_0^{+\omega}) \ast \E(\mu_0^{+\omega+1}, k(\mu_0))$ such that $p \Vdash |\dot{f} \image A^{<\omega} \cap k(\mu_0^{+n})| = \aleph_1$.

Let us find a sequence of decreasing conditions $p_\alpha$ below $p$ and a sequence of sets $a_\alpha \in (\mu_0^{+\omega+1})^{{<}\omega}$ such that $p_\beta \Vdash \dot f(k(a_\alpha)) < \dot f(k(a_\beta)) < k(\mu_0^{+n})$ for every $\alpha < \beta < \mu_0^{+\omega+1}$. We find this sequence in the same way as we did in Theorem \ref{thm: cc at single point}. Namely, let $\langle \dot{x}_i \mid i < \mu_0^{+\omega+1}\rangle$ be a sequence of names for elements in $A^{<\omega}$ such that $p \Vdash \dot{f}(\dot{x}_\alpha) < \dot{f}(\dot{x}_\beta) < k(\mu_0^{+n})$ for every $\alpha < \beta$. Let us pick for every $\alpha < \mu_0^{+\omega + 1}$ a condition $p_\alpha = \langle r_\alpha, \dot{q}_\alpha\rangle$ below $p$ such that:
\begin{enumerate}
\item For some $a_i \in (\mu_0^{+\omega+1})^{{<}\omega}$, $p_\alpha \Vdash \dot{x}_\alpha = k(\check a_\alpha)$.
\item For every $\beta < \alpha$, $r_0 \Vdash \dot{q}_\alpha \leq \dot{q}_\beta$.
\end{enumerate}

Given the partial sequence $\langle r_\alpha, \dot q_\alpha, a_\alpha : \alpha<\beta \rangle$ satisfying the above conditions, we let $\dot q$ be a name for a lower bound to $\langle \dot q_\alpha : \alpha < \beta \rangle$.  Then we pick $\langle r_\beta, \dot q' \rangle \leq \langle r_0, \dot q \rangle$ and $a_\beta$ such that $\langle r_\beta, \dot q' \rangle \Vdash \dot x_\beta = k(\check a_\beta)$.  Then let $\dot q_\beta$ be such that $r_\beta \Vdash \dot q_\beta = \dot q'$ and $r' \Vdash \dot q_\beta = \dot q$ for all $r' \perp r_\beta$.
By the regularity of $\mu_0^{+\omega + 1}$, there is a fixed condition $r_\star$ and a cofinal set of ordinals $\alpha < \mu_0^{+\omega+1}$ such that $r_\alpha = r_\star$. Without loss of generality, for every $\alpha$, $r_\alpha = r_\star$.

By elementarity, for every $\alpha < \beta < \mu_0^{+\omega + 1}$ there is $\rho < \mu_0$ and a condition $q \in \Col(\omega, \rho^{+\omega}) * \mathbb E(\rho^{+\omega+1}, \mu_0)$ that forces $\dot f_\rho(a_\alpha) < \dot f_\rho(a_\beta) < \mu_0^{+n}$. Applying the \Erdos-Rado theorem on the first $\mu_0^{+n + 1}$ elements in this sequence, we obtain a sequence of ordinals $I$ of order type $\mu_0^{+n} + 1$ and a single $\rho_\star < \mu_0$, $q_\star \in \Col(\omega, \rho_\star^{+\omega}) * \mathbb E(\rho_\star^{+\omega+1}, \mu_0)$ such that for every $\alpha < \beta$ in $I$, $q_\star \Vdash \dot f_{\rho_\star}(a_\alpha) < \dot f_{\rho_\star}(a_\beta) < \mu_0^{+n}$. This is a contradiction, since it is impossible to get an increasing sequence of ordinals of length $\mu_0^{+n}+1$ below $\mu_0^{+n}$.

Thus, there is $\rho < \mu_0$ such that the product forces $(\aleph_{\omega+1}, \aleph_{\omega}) \chang (\aleph_1, \aleph_0)$.  Let us show that in this case the iteration does the same.

\begin{claim}
There is a projection
\begin{align*}
\pi & : (\Col(\omega, \rho^{+\omega}) \ast \E(\rho^{+\omega+1}, \mu_0)) \times \E(\mu_0, \mu_1) \times \cdots \times \E(\mu_n, \mu_{n+1})\times \cdots  \\
 & \to \Col(\omega, \rho^{+\omega}) \ast \E(\rho^{+\omega+1}, \mu_0) \ast \E(\mu_0, \mu_1) \ast \cdots \ast \E(\mu_n, \mu_{n+1})\ast \cdots
\end{align*}
\end{claim}
\begin{proof}
Let $\mathbb P = \Col(\omega, \rho^{+\omega}) \ast \E(\rho^{+\omega+1}, \mu_0)$. The argument for Lemma \ref{easton projection} shows the following:  For each $n < \omega$, there is a map $$\sigma_n : \E(\mu_{n},\mu_{n+1}) \to T(\mathbb P \ast \E(\mu_0, \mu_1) \ast \cdots \ast \E(\mu_{n-1}, \mu_{n}),  \E(\mu_{n},\mu_{n+1}))$$ such that $\langle p,q \rangle \mapsto \langle p, \sigma_n(q) \rangle$ is a projection from $(\mathbb P \ast \E(\mu_0, \mu_1) \ast \cdots \ast \E(\mu_{n-1}, \mu_{n})) \times \E(\mu_n,\mu_{n+1})$ to $(\mathbb P \ast \E(\mu_0, \mu_1) \ast \cdots \ast \E(\mu_{n-1}, \mu_{n}) \ast \E(\mu_{n}, \mu_{n+1})$.  Furthermore, if $p \Vdash \dot q_1 \leq \sigma_n(\check q_0)$, then there is $q_2 \leq q_0$ such that
$p \Vdash \sigma_n(\check q_2) = \dot q_1$.

For a condition $r = \langle p,q_0,q_1,\dots \rangle$ in the infinite product we define $\pi(r) = \langle p, \sigma_0(q_0),\sigma_1(q_1),\dots \rangle$.  To verify that $\pi$ is a projection, suppose $\langle p',q'_0,q'_1,\dots \rangle \leq \langle p, \sigma_0(q_0),\sigma_1(q_1),\dots \rangle.$  For each $n$, there is $q''_n \leq q_n$ such that $\langle p',q'_0,\dots,q'_{n-1} \rangle \Vdash \sigma_n(q''_n) = q'_n$.  An easy induction argument shows that $\langle p', \sigma_0(q''_0),\sigma_1(q''_1),\dots \rangle \leq \langle p',q'_0,q'_1,\dots \rangle$.
\end{proof}

As the product is $\mu_0$-closed in the ground model, it is $\mu_0$-distributive in the generic extension by the $\mu_0$-c.c.\ forcing $\Col(\omega, \rho^{+\omega}) \ast \E(\rho^{+\omega+1}, \mu_0)$. Therefore, if $f : \aleph_{\omega+1}^{{<}\omega} \to \aleph_\omega$ is in the extension by the iteration, the $\aleph_1$-sized witness for Chang's conjecture with respect to $f$ already exists in the extension by the iteration.
\end{proof}

\section{Open Questions}
We conclude this paper with a list of open questions:
\begin{question} Is it consistent, relative to large cardinals, that for every pair of cardinals $\kappa < \lambda$ such that $\kappa < \cf \lambda$ or $\cf \kappa = \cf \lambda$, \[(\lambda^+, \lambda)\chang (\kappa^+,\kappa)?\]
\end{question}
In the model of Section~\ref{sec: global cc} we gave a positive answer to this question when restricting $\lambda$ to be a successor cardinal.

\begin{question}
What is the consistency strength of $(\aleph_4, \aleph_3) \chang (\aleph_2, \aleph_1)$?
\end{question}
In Section~\ref{sec: local cc} we gave an upper bound of $(+2)$-subcompact cardinal. The known lower bound, due to Levinski is $0^\dagger$, \cite{Levinski1984}.
\providecommand{\bysame}{\leavevmode\hbox to3em{\hrulefill}\thinspace}
\providecommand{\MR}{\relax\ifhmode\unskip\space\fi MR }
\providecommand{\MRhref}[2]{%
  \href{http://www.ams.org/mathscinet-getitem?mr=#1}{#2}
}
\providecommand{\href}[2]{#2}

\end{document}